\documentclass[draft,reqno]{amsart}
\usepackage{mathtools}
\usepackage{amssymb}
\usepackage{hyperref}
\usepackage{cleveref} %Shall be loaded after hyperref. 

\usepackage{enumitem}

\usepackage{esint} %Provides extended collection of integral symbols, including the symbol for average integral $\fint$.

\usepackage{dutchcal}%Provides lowercase \mathcal while overwriting uppercase \mathcal.
\allowdisplaybreaks

\renewcommand{\phi}{\varphi}
\renewcommand{\epsilon}{\varepsilon}
\renewcommand{\geq}{\geqslant}
\renewcommand{\leq}{\leqslant}

\renewcommand{\P}{\mathbb{P}}
\renewcommand{\Im}{\operatorname{Im}}
\renewcommand{\Re}{\operatorname{Re}}

\newcommand{\R}{\mathbb{R}}
\newcommand{\C}{\mathbb{C}}

\newcommand{\T}{\mathbb{T}}
\newcommand{\x}{\boldsymbol{x}}
\newcommand{\y}{\boldsymbol{y}}

\newcommand{\vv}{\boldsymbol{v}}

\newcommand{\oomega}{\boldsymbol{\omega}}
\newcommand{\supp}{\operatorname{supp}}
\newcommand{\dist}{\operatorname{dist}}
\newcommand{\nb}{\mathcal{N}} %neighbourhood
\newcommand{\Q}{\mathcal{Q}}
\newcommand{\q}{\mathcal{q}}

\newcommand{\Span}{\operatorname{span}}

\newtheorem{theorem}{Theorem}[section]
\newtheorem{definition}[theorem]{Definition}

\newtheorem{proposition}[theorem]{Proposition}
\newtheorem{lemma}[theorem]{Lemma}
\newtheorem{remark}[theorem]{Remark}
\newtheorem{corollary}[theorem]{Corollary}

\newtheorem{question}[theorem]{Question}
\newtheorem{fact}[theorem]{Fact}
\newtheorem{claim}[theorem]{Claim}

\DeclareFontFamily{U}{mathx}{\hyphenchar\font45}
\DeclareFontShape{U}{mathx}{m}{n}{
      <5> <6> <7> <8> <9> <10>
      <10.95> <12> <14.4> <17.28> <20.74> <24.88>
      mathx10
      }{}
\DeclareSymbolFont{mathx}{U}{mathx}{m}{n}
\DeclareFontSubstitution{U}{mathx}{m}{n}
\DeclareMathAccent{\widecheck}{0}{mathx}{"71}
\newcommand{\dbtilde}[1]{\accentset{\approx}{#1}}
\usepackage{accents}
\makeatletter
\newcommand{\sbullet}{%
  \hbox{\fontfamily{lmr}\fontsize{.4\dimexpr(\f@size pt)}{0}\selectfont\textbullet}}
\DeclareRobustCommand{\mathbullet}{\accentset{\sbullet}}
\makeatother
\title{Incidence estimates for tubes in complex space}

\author{Sarah Tammen}
\address{Department of Mathematics\\
University of Wisconsin - Madison\\
Madison, WI 53706}
\email{tammen2@wisc.edu}

\author{Lingxian Zhang}
\address{Department of Mathematics and International Center for Mathematics, Southern
University of Science and Technology, Shenzhen 518055, China}
\email{l.zhang.work1@gmail.com}


\date{}

\begin{document}

\begin{abstract}
In this paper, we prove a complex version of the incidence estimate of Guth, Solomon and Wang \cite{WellSpacedTubes} for tubes obeying certain strong spacing conditions, and we use one of our new estimates to resolve a discretized variant of Falconer's distance set problem in $\C^2$.
\end{abstract}

\maketitle

\section{Introduction}
It has long been known that in Fourier analysis and related fields, only some of the results true in the real space hold in the complex space as well. For instance, 
\begin{itemize}
\item the Kakeya problem in dimension three under the Wolff axioms is true over $\R$ \cite{Kakeya_3dim_WolffAxioms_Real_new} but false over $\C$ because of the Heisenberg example \cite{Kakeya_3dim_WolffAxioms_Old}; 
\item the Falconer distance problem and its cousins, the Erd\"os ring problem and the Furstenberg problem, have counterexamples over $\C$ because of the existence of the quadratic subfield $\R$ \cite{WolffRecentKakeya}, but are still open over $\R$; 
\item the Szemer\'edi-Trotter theorem is true over both $\R$ \cite{SzT_Original} and $\C$ \cite{SzT_Complex}; 
\item the Erd\H{o}s distinct distances problem, with an $\epsilon$-loss, is true over both $\R$ \cite{ErdosDistinctDist_Real} and $\C$ \cite{ErdosDistinctDist_Complex}.
\end{itemize}
The results in this paper are complex analogues of the incidence estimates and distance estimates obtained in \cite{WellSpacedTubes}, giving some more examples of problems that have positive answers over both $\R$ and $\C$.
%--Add statements of theorems in the real case?
%--Maybe not.

\par By ``complex lines", we will mean the translations of $1$-dimensional vector subspaces of $\C^n$ (over the field $\C$), and ``complex tubes" will refer to small neighborhoods of complex line segments. By ``the angle between the tubes", we will mean the angle between the line segments generating the tubes. One intricacy of this paper is in defining the angle between any two complex directions. There are two common ways to do so, one via an explicit metric on $\C\P^{n-1}$, and another via a variational formula in $\R^{2n}$. The two turn out to be equivalent, and details are spelled out in \Cref{sect_def}. We will implicitly invoke the metric definition when discussing the partition of $\C\P^{n-1}$ into tiny parts. 

\par Some more terminology we will be using: A \emph{$\delta$-ball} is a ball of radius $\delta$, a \emph{$\delta$-tube} is the $\delta$-neighborhood of a unit-length line segment. We call a subset $\theta\subset\C\P^{n-1}$ an \emph{almost-$\delta$-cap} if there exists a point $\boldsymbol{u}\in\C\P^{n-1}$ such that \[B_{\C\P^{n-1}}\left(\boldsymbol{u},\tfrac{1}{2}\delta\right)\subseteq\theta\subseteq B_{\C\P^{n-1}}\left(\boldsymbol{u},2\delta\right).\]
We say two solids $A_1,A_2\subseteq\C^n\simeq\R^{2n}$ (e.g. balls, complex tubes, etc.) essentially intersect if their intersection is at least half as large as the maximum intersection between all rigid transformations of the objects, i.e. \[\left|A_1\cap A_2\right|\geq\tfrac{1}{2}\max_{\sigma_1,\sigma_2\in\operatorname{SE}(2n)}\left|\sigma_1\left(A_1\right)\cap\sigma_2\left(A_2\right)\right|.\]
Two solids of the same kind are said to be essentially distinct if they do not essentially intersect, or else they are said to be essentially the same. For example, two $\delta$-tubes $T_1$ and $T_2$ are essentially distinct if the volume of $T_1 \cap T_2$ is no more than half the volume of $T_1$. One solid $A_1$ is said to essentially contain another solid $A_2$ if $A_1$ contains at least half of $A_2$, i.e. \[|A_2\setminus A_1|<\tfrac{1}{2}|A_2|.\]

\par Following common conventions, we will write $X\lesssim_{\alpha} Y$ to mean that there exists a positive constant $C_{\alpha,n}$ depending only on the parameter $\alpha$ and the dimension $n$ such that $X\leq C_{\alpha,n} Y$. The dependency on the dimension $n$ is often suppressed in the notation. We will write $X\sim_{\alpha} Y$ to mean that $X\lesssim_{\alpha} Y\lesssim_{\alpha} X$. Plus, we will write $X(\delta)\ll Y(\delta)$ if $\lim_{\delta\to 0}\frac{X(\delta)}{Y(\delta)}=0$. 

\par Given a collection $\T$ of solids of the same kind (e.g. complex tubes), we denote by $P_r(\T)$ a maximal\footnote{That is to say, we first consider the collection of all the $\delta$-balls that are essentially contained in the intersection of at least $r$ of the solids $T\in\T$, i.e. $r$-rich, and then take a subcollection $P_r(\T)$ consisting of essentially distinct $\delta$-balls such that every other $r$-rich $\delta$-ball in the super-collection is essentially the same as one of the $\delta$-balls in the subcollection. There can be a little ambiguity as to what exactly this subcollection $P_r(\T)$ is; but since all the maximal subcollections have about the same size, we just fix any one of them.} collection of essentially distinct $r$-rich $\delta$-balls.
\par Now we are ready to describe our main theorems. The precise spacing condition will be stated later when we prove the theorems. (See the statements of \Cref{thm:1st_main}, \Cref{thm:2nd_main} and \Cref{thm:diff_set_size}, which corresponds to \Cref{thm:intro_1st_main}, \Cref{thm:intro_2nd_main} and \Cref{cor:intro_diff_set_size}, respectively.) 

%---------Main Theorems------------
\begin{theorem}\label{thm:intro_1st_main}
Let $\Theta$ be a partition of $\C\P^{n-1}$ into almost-$\delta$-caps. Suppose $1\leq W\leq\delta^{-1}$ and $1\leq N\leq W^{-1}\delta^{-1}$. For each $\theta\in\Theta$, let $\left\{T_{\theta,j}\right\}_{1\leq j\lesssim NW^{2(n-1)}}$ be a family of essentially distinct complex $\delta$-tubes contained in the unit ball $B_{\C^n}(\boldsymbol{0},1)$ with direction in~$\theta$ such that every complex $W^{-1}$-tube contained in the unit ball with direction in~$\theta$ essentially contains~$N$ of the complex tubes~$T_{\theta,j}$. Let~$\T$ denote the set of all the tubes~$T_{\theta,j}$. Then, for richness $r\gtrsim_{\epsilon}\delta^{2(n-1)-\epsilon}|\T|$, \[|P_r(\T)|\lesssim_{\epsilon}\delta^{-\epsilon}W^{-2(n-1)}r^{-2}|\T|^2.\]
\end{theorem}

\begin{theorem}\label{thm:intro_2nd_main}
Let $n=2$~or~$3$, $1\leq W\leq\delta^{-1}$, and $\T$ be a collection of complex $\delta$-tubes essentially contained in $B_{\C^n}(\boldsymbol{0},1)$. Suppose each complex $W^{-1}$-tube essentially contained in $B_{\C^n}(\boldsymbol{0},1)$ contains $N_0$ complex $\delta$-tube $T\in\T$. Then, for $r\geq\max\{ \delta^{2(n-1)-\epsilon}|\T|,1+N_0\}$, \[|P_r(\T)|\lesssim_{\epsilon,n,N} \delta^{-\epsilon}|\T|^{\frac{n}{n-1}}r^{-\frac{n+1}{n-1}}.\]
\end{theorem}

\begin{corollary}\label{cor:intro_diff_set_size}
Fix $2<s<4$. Let $E$ be a set of $\sim\delta^{-s}$ many $\delta$-balls in $B_{\C^2}(\boldsymbol{0},1)$ with at most $N$ $\delta$-balls in each ball of radius $\delta^{s/4}$. Then, the number of disjoint $\delta$-balls needed to cover the difference set \[\Delta(E):=\left\{(x_1-x_2)^2+(y_1-y_2)^2\in\C:\Vec{p}_1=(x_1,y_1),\Vec{p}_2=(x_2,y_2)\in\bigcup E\right\}\]
is $\gtrsim_{\epsilon,s,N}\delta^{-s+\epsilon}$ for all $\epsilon>0$. 
\end{corollary}

\par In \Cref{sect_def}, we review some definitions and conventions that will be used in the rest of the paper. In \Cref{sect_lem}, we prove a lemma which allows us to capture the behavior of a large number of small balls using a few slightly larger balls, and thus lays the foundation for induction on scale. In \Cref{sect_thm}, we consider two strong spacing conditions on complex tubes, and prove an upper bound on the number of rich balls in each case. And lastly, in \Cref{sect_app}, we derive a lower bound on the number of distinct complex distances for certain sparse sets in $\C^2$ by reducing the problem to one about incidences between complex tubes.

\section{Definitions, notations and conventions}\label{sect_def}
%---- Definition of complex angle----
As mentioned in the introduction, there are two equivalent\footnote{To see the equivalence, write $\boldsymbol{u}_1=(a_1+b_1i,c_1+d_1i)$ and $\boldsymbol{u}_2=(a_2+b_2i,c_2+d_2i)$, with $a_1,a_2,b_1,b_2,c_1,c_2,d_1,d_2\in\R$. Then, for any $z=e^{i\zeta}$ with $\zeta\in\R$, \begin{align*}
\iota(\boldsymbol{u}_1)\cdot\iota(z\boldsymbol{u}_2)=(a_1a_2+b_1b_2+c_1c_2+d_1d_2)\cos\zeta+(-a_1b_2+a_2b_1-c_1d_2+c_2d_1)\sin\zeta.
\end{align*}
Hence, \begin{align*}
&\min_{z_1, z_2 \in \C: |z_1|=|z_2| = 1} \arccos \left( \iota(z_1 \boldsymbol{u}_1) \cdot \iota(z_2 \boldsymbol{u}_2) \right)\\
&=\arccos \left(\max_{z\in\C:|z|=1}\iota(\boldsymbol{u}_1) \cdot \iota(z \boldsymbol{u}_2) \right)\\
&=\arccos\sqrt{(a_1a_2+b_1b_2+c_1c_2+d_1d_2)^2+(-a_1b_2+a_2b_1-c_1d_2+c_2d_1)^2}\\
&=\arccos| \langle \boldsymbol{u}_1, \boldsymbol{u}_2 \rangle|
\end{align*}} ways to define the angle between two complex lines, say \[\ell_1=\left\{z\boldsymbol{u}_1+\boldsymbol{t}_1\in\C^n:z\in\C\right\}\quad\text{and}\quad\ell_2=\left\{z\boldsymbol{u}_2+\boldsymbol{t}_2\in\C^n:z\in\C\right\},\]
with $\|\boldsymbol{u}_1\|=\|\boldsymbol{u}_2\|=1$: \begin{enumerate}
\item Treat the complex lines as $2$-planes in $\R^{2n}$, and define the angle between the complex lines to be the first principal angle between the planes:
\[
\min_{z_1, z_2 \in \C: |z_1|=|z_2| = 1} \arccos \left( \iota(z_1 \boldsymbol{u}_1) \cdot \iota(z_2 \boldsymbol{u}_2) \right),
\]
where $\iota(\boldsymbol{u})$ denotes the usual embedding of $\boldsymbol{u} \in \C^n$ into $\R^{2n}$;
\item Define the angle numerically to be $\arccos| \langle \boldsymbol{u}_1, \boldsymbol{u}_2 \rangle|$, where $\langle \cdot , \cdot \rangle$ is the usual inner product on $\C^n$; this is also the Fubini-Study distance, up to a scaling constant, between the equivalent classes of $\boldsymbol{u}_1$ and $\boldsymbol{u}_2$ in $\C\P^{n-1}$. 
\end{enumerate}
Note that for $2$-planes in $\R^{2n}$ corresponding to complex tubes, the two principal angles are equal\footnote{Let's assume WLOG that $\boldsymbol{t}_1=\boldsymbol{t}_2=\boldsymbol{0}$. Suppose the first principal angle is attained for the pair of unit vectors $\left(\boldsymbol{v}_1,\boldsymbol{v}_2\right)\in e^{i\R}\boldsymbol{u}_1\times e^{i\R}\boldsymbol{u}_2$. Since the orthogonal complement of $\iota\left(\boldsymbol{v}_j\right)$ in $\iota\left(\ell_j\right)$ is $\Span_{\R}\left(\iota\left(i\boldsymbol{v}_j\right)\right)$ for each $j\in\{1,2\}$, the second principal angle between $\iota\left(\ell_1\right)$ and $\iota\left(\ell_2\right)$ is \[\arccos\left(\iota\left(i\boldsymbol{v}_1\right)\cdot\iota\left(i\boldsymbol{v}_2\right)\right)=\arccos\left(\iota\left(\boldsymbol{v}_1\right)\cdot\iota\left(\boldsymbol{v}_2\right)\right). \]}. Hence the study of incidences between complex lines in $\C^n$ is different from the study of incidences between general $2$-planes in $\R^{2n}$.

%----Definition of complex tubes---- MOVED TO INTRO!

%----Essential distinctness---- MOVED TO INTRO!

%----Volume of the intersection of two complex tubes---
\par The following proposition implies that two tubes with the same centre are essentially distinct only if the angle between them is $\gtrsim \delta$.  More generally, if tubes $T_1$ and $T_2$ intersect with 
\[
|T_1 \cap T_2| > \tfrac{1}{2} \max_{\boldsymbol{t} \in \C^n} \Big(|T_1 \cap(T_2 + \boldsymbol{t})| \Big),
\]
then $T_1$ and $T_2$ are essentially distinct only if the angle between $T_1$ and $T_2$ is $\gtrsim \delta$.

\begin{proposition}\label{prop:intersection-of-complex-tubes}
Consider $\delta$-neighborhoods of the lines
\[ \ell_1 = \{ z \boldsymbol{u}_1 : z \in \C \}\quad\text{and}\quad
\ell_2 = \{ z \boldsymbol{u}_2 : z \in \C \}\]
for some direction vectors $\boldsymbol{u}_1$ and $\boldsymbol{u}_2$ that form angle $\theta > 0$ to each other, according to the definition above. Let $\iota: \C^n \to \R^{2n}$ be the usual embedding. Then 
\begin{equation}
\label{intersection volume}
\left|\iota(N_{\delta}(\ell_1)) \cap \iota(N_{\delta}(\ell_2))\right| \sim \frac{\delta^{2n}}{\sin^2 \theta}.
\end{equation}
\end{proposition}

We prove this proposition in the appendix.

%----Scaling shapes----
\par Given any shape $A\subseteq\C^n\simeq\R^{2n}$ and any scalar $b\in[0,\infty)$, we denote by $bA$ the uniformly scaled shape $\left\{b(\boldsymbol{\alpha}-\boldsymbol{c}_A)+\boldsymbol{c}_A:\boldsymbol{\alpha}\in A\right\}$, with $\boldsymbol{c}_A$ being the centre of mass of $A$.
%----Dual shape---
\par  Given any shape $A\subseteq\C^{n}$, we define its dual, denoted $A^*$, as follows: \[ A^*:=\left\{\x\in\C^{n}:\left|\x\cdot\left(\y-\boldsymbol{c}_A\right)\right|<1\text{ for all }\y\in A\right\}.\]
For example, consider a tube of radius $\sim\delta$ and length $\sim1$ in $\R^n$, not necessarily centred at the origin. Its dual is a slab of radius $\sim\delta^{-1}$ and thickness $\sim 1$ centred at the origin. 
%----Uniform family of bump functions----
\par By fixing a bump function for the unit ball, and then pre-composing the bump with translation and isotropic scaling (and rotation if we wish), we can construct smooth bump functions for balls in $\C^n$ in a uniform way. In a similar manner, we can uniformly construct bump functions for complex tubes as well. These uniformly constructed bump functions will be called ``\emph{the} bump functions" for the balls and the tubes.
\section{A bridge between different scales}\label{sect_lem}
To study the high-frequency part of a later integral, we will need the following geometric fact. 
%-----Important geometric fact------
\begin{fact}\label{fact:geom_fact} Let $\T$ be any collection of essentially distinct complex tubes of length $\sim 1$ and radius $\sim D$, and $\oomega$ be any point in $B_{\C^n}(\boldsymbol{0},\lambda)\setminus B_{\C^n}(\boldsymbol{0},\rho)$, with $\lambda D^{-1}\ll\lambda^{-1}\ll\rho\ll1$ (say $\lambda:=D^{\frac{\epsilon}{100n}}$ and $\rho:=D^{\epsilon^3}\lambda^{-1}$). Then, the number of enlarged dual complex slabs $\lambda T^*$ (centred at the
origin of width $\sim\lambda$ and thickness $\sim\lambda D^{-1}$) passing through the point $\oomega$, not counting any multiplicity\footnote{For example, imagine two complex tubes $T_1$ and $T_2$ such that $T_2$ is a translate of a tube that is not essentially distinct from $T_1$. Then, the two dual complex slabs $T_1^*$ and $T_2^*$ are not essentially distinct. If both $\lambda T_1^*$ and $\lambda T_2^*$ pass through the point $\oomega$, they would count as one enlarged dual slab rather than two.}, is \[\lesssim\lambda^2\|\oomega\|^{-2}D^{2(n-2)}.\]
\end{fact}

\begin{proof}
It's not hard to see from symmetry that the maximum is no more than a constant multiple of the average number of enlarged dual slabs in a maximal $D^{-1}$-separated collection that pass through a point on the sphere $\partial B_{\C^n}(\boldsymbol{0},\|\oomega\|)$. A double-counting argument then gives \begin{align*}
\#\{\text{enlarged dual slabs passing through }\oomega\}&\sim\frac{D^{2(n-1)}\cdot\left(\lambda D^{-1}\right)^2\|\oomega\|^{2(n-2)} }{\|\oomega\|^{2(n-1)}}\\
&=\lambda^2\|\oomega\|^{-2}D^{2(n-2)}.\hfill\qedhere
\end{align*}
\end{proof}

%-----Lemma: Finding Complex Heavy Balls------
In the proof below, we do not make a clear distinction between integrals over (subsets of) $\C^n$ and integrals over (subsets of) $\R^{2n}$. The readers shall treat all integrals as over $\R^{2n}$, especially when Fourier transforms are involved.

\begin{lemma}\label{lem:heavy_ball_lemma}
Suppose $P$ is a set of essentially distinct unit balls in $B_{\C^n}(0,D)$, and $\T$ is a set of essentially distinct tubes of length $D$ and radius $1$ in $B_{\C^n}(0,D)$ such that each ball of $P$ lies in $\sim E$ (more specifically, at least $E$ and less than $2E$) tubes of $\T$, Then, for any small $\epsilon>0$ and $1\ll \rho^{-1}\ll\lambda\ll D$ (say $\lambda:=D^{\frac{\epsilon}{100n}}$ and $\rho:=D^{\epsilon^3}\lambda^{-1}$). at least one of the following is true:
\begin{enumerate}
\item (Thin case) $|P|\lesssim(\log D)\lambda^{2}\rho^{-2}E^{-2}D^{2(n-1)}|\T|$.
\item (Thick case) There is a set of finitely overlapping $2\lambda$-balls $Q_j$ such that \begin{enumerate}
\item $\bigcup_j Q_j$ contains a $\gtrsim(\log D)^{-1}$ fraction of the balls $q\in P$; and
\item each $Q_j$ intersects $\gtrsim E\lambda^{-2}\rho^{-2n}$ tubes $T\in\T$.
\end{enumerate}
\end{enumerate}
\end{lemma}

%-----Pf of Lem: High-low frequency------
\begin{proof}
Let's begin with some pigeonholing. For each dyadic $E\lesssim k\lesssim D^{2(n-1)}$, let $P_k\subset P$ be the set of balls contained in $\sim k$ of the thickened tubes $\nb_\lambda(T)$. Then, there exists an index $k_0$ for which \[|P_{k_0}|\gtrsim_n(\log D)^{-1}|P|.\]
We will study the subset $P':=P_{k_0}$ instead of~$P$.
\par Define \[f:=\sum_{q\in P'}\psi_q,\qquad g:=\sum_{T\in\T}\psi_T,\]
and \[I(P',\T):=\left\{(q,T)\in P'\times\T:|q\cap T|\geq\tfrac{1}{2}|q|\right\}.\]
Then, treating $f$ and $g$ as functions on $\R^{2n}$, \[I(P',\T)\lesssim\int fg=\int\widehat{f}\,\overline{\widehat{g}}=\int\eta\widehat{f}\,\overline{\widehat{g}}+\int(1-\eta)\widehat{f}\,\overline{\widehat{g}},\]
where $\eta$ is ``the'' smooth bump function approximating $\chi_{B(\boldsymbol{0},\rho)}$.

\par In the low frequency case, \begin{align*}
I(P',\T)\lesssim\int\eta\widehat{f}\,\overline{\widehat{g}}=\int f\left(g*\widecheck{\eta}\right)=\sum_{q\in P'}\sum_{T\in\T}\int\psi_q\left(\psi_T*\widecheck{\eta}\right).
\end{align*}
Since $\widecheck{\eta}$ decays rapidly outside of $B\left(\boldsymbol{0},\rho^{-1}\right)$, $\lambda\gg\rho^{-1}\gg 1$ and $\supp\psi_T\subseteq 2T$, \[\left|\psi_T*\widecheck{\eta}\right|\lesssim\sup_{\oomega\in\C^n}\left|B(\oomega,\lambda)\cap 2T\right|\cdot\left\|\widecheck{\eta}\right\|_{L^\infty}\sim\lambda^2\cdot\rho^{2n}\]
Hence, for each $q\in P'$, \[\sum_{T\in\T}\int\psi_q\left(\psi_T*\widecheck{\eta}\right)\lesssim\lambda^2\rho^{2n}\cdot\#\left\{T\in\T: q\cap\nb_{\lambda}(T)\neq\varnothing\right\}.\]
It follows that \[E|P'|\sim I(P',\T)\lesssim\sum_{q\in P'}\lambda^2\rho^{2n}\cdot\#\left\{T\in\T: q\cap\nb_{\lambda}(T)\neq\varnothing\right\}.\]
Thus, for balls $q\in P'=P_{k_0}$, we have \[\#\left\{T\in\T: q\cap\nb_{\lambda}(T)\neq\varnothing\right\}\gtrsim E\lambda^{-2}\rho^{-2n}.\]
Note that \[\left\{T\in\T: q\cap\nb_{\lambda}(T)\neq\varnothing\right\}=\left\{T\in\T: T\cap\nb_{\lambda}(q)\neq\varnothing\right\},\]
and each neighbourhood $\nb_{\lambda}(q)$ can be covered by a $2\lambda$-ball. 

\par As for the high frequency case, \begin{align*}
I(P',\T)\lesssim\int(1-\eta)\widehat{f}\,\overline{\widehat{g}}\leq \left(\int(1-\eta)\left|\widehat{f}\right|^2\right)^{\!\frac{1}{2}}\left(\int(1-\eta)\left|\widehat{g}\right|^2\right)^{\!\frac{1}{2}}
\end{align*}
where \[\left(\int(1-\eta)\left|\widehat{f}\right|^2\right)^{\!\frac{1}{2}}\leq\left\|\widehat{f}\right\|_{L^2}=\left\|f\right\|_{L^2}\sim|P'|^{\frac{1}{2}}.\]
Group the tubes in $\T$ according to their directions: Let $\Theta$ be a partition of $\C\P^{n-1}$ into almost-$D^{-1}$-caps $\theta$. Define $\T_\theta$ to be the set of tubes in $\T$ with direction in $\theta$, and set \[g_\theta:=\sum_{T\in\T_\theta}\psi_T.\]
For each $\theta\in\Theta$, fix a direction $\vv_{\theta}\in\theta$ and a tube $T_\theta$ centred at the origin of length $\sim 1$ and radius $~\sim D^{-1}$ in direction~$\vv_{\theta}$ which is contained in all tubes centred at the origin of length $\sim 1$ and radius $~\sim D^{-1}$ with direction in $\theta$.
Then, for each $\theta\in\Theta$ and $T\in\T_{\theta}$, $\widehat{\psi_T}$ decays rapidly outside of $T_{\theta}^*$. So, \begin{align*}
&\int(1-\eta)\left|\widehat{g}\right|^2\\
&= \int(1-\eta)\left|\sum_{\theta\in\Theta}\widehat{g_\theta}\right|^2\\
&\sim\int(1-\eta(\oomega))\left|\sum_{\theta:\lambda T_\theta^*\ni\oomega}\widehat{g_\theta}(\oomega)\right|^2 d\oomega+\int(1-\eta(\oomega))\left|\sum_{\theta:\lambda T_\theta^*\not\ni\oomega}\widehat{g_\theta}(\oomega)\right|^2 d\oomega\\
&\lesssim\int_{\C^n\setminus B_{\C^n}(\boldsymbol{0},\rho)}(1-\eta(\oomega))\cdot\#\{\theta\in\Theta:\lambda T_\theta^*\ni\oomega \}\cdot\sum_{\theta\in\Theta}\left|\widehat{g_\theta}(\oomega)\right|^2 d\oomega\\
&\lesssim \lambda^{2}\rho^{-2}D^{2(n-2)}\sum_{\theta\in\Theta}\int\left|\widehat{g_\theta}\right|^2.
\end{align*}
where, in the last inequality, we used \Cref{fact:geom_fact}, and \[\sum_{\theta\in\Theta}\int\left|\widehat{g_\theta}\right|^2=\sum_{\theta\in\Theta}\int\left|g_\theta\right|^2\lesssim\sum_{T\in\T}\int\left|\psi_T\right|^2\sim|\T|\cdot D^2.\]
Putting things together, we get \[E|P'|\sim I(P',\T)\lesssim|P|^{\frac{1}{2}}\cdot \lambda\rho^{-1}D^{n-2}\cdot|\T|^\frac{1}{2}D,\]
which reduces to \[|P'|\lesssim E^{-2}\lambda^2\rho^{-2}D^{2(n-1)}|\T|. \hfill\qedhere\]
\end{proof}

\section{Proof of the main theorems}\label{sect_thm}
We rephrase here \Cref{thm:intro_1st_main}.
\begin{theorem}\label{thm:1st_main} Suppose $1\leq W\leq\delta^{-1}$ and $1\leq N\leq W^{-1}\delta^{-1}$. Let $\Theta$ be a partition of $\C\P^{n-1}$ into almost-caps $\theta$ of radius $\sim\delta$. And for each $\theta\in\Theta$, let $\left\{T_{\theta,j}:1\leq j\lesssim W^{2(n-1)}N\right\}$ be a family of essentially distinct complex tubes of radius $\delta$ and length $1$ essentially contained in $B_{\C^n}(\boldsymbol{0},1)$ with the property that, for each direction $\theta$ and for a fixed maximal collection of essentially distinct complex tube of radius $W^{-1}$ and length $1$ essentially contained in $B_{\C^n}(\boldsymbol{0},1)$ with direction in~$\theta$, each of the complex $W^{-1}$-tube contains, essentially, $\sim N$ of the tubes $T_{\theta,j}$. Let $\T$ denote the set of all the tubes $T_{\theta,j}$. Then, for each $\epsilon>0$, there exist constants $c_1(\epsilon)$ and $c_2(\epsilon)$ such that, for $r\geq c_1(\epsilon)\,\delta^{2(n-1)-\epsilon}|\T|$, \[|P_r(\T)|\leq c_2(\epsilon)\,\delta^{-\epsilon}W^{-2(n-1)}r^{-2}|\T|^2.\]
\end{theorem}
%The constant $c_1$ is meant to be large, so that we have a margin above the average; $c_2$ is large as well so that we can establish a base case for $\delta$ not too tiny.

%-----Pf of 1st Thm: Induction on Scale-------
\begin{proof}
Let's start with two base cases: \begin{enumerate}[label=(\Roman*)]
\item If $c_0(n,\epsilon)<\delta<1$, where $c_0$ is a sufficiently small\footnote{Being sufficiently small here means being small enough that all the expressions of the form $\delta^{\epsilon\cdot h(\epsilon)}$, with $h(\epsilon)$ a real-coefficient polynomial, appearing later in this section are as small or as large as needed.} constant, then $P_r(\T)\neq\varnothing$ only if $r\lesssim1$, under which condition both $|P_r(\T)|$ and $W^{-2(n-1)}r^{-2}|\T|^2$ are~$\sim 1$.
\item If $r\geq\alpha_n\,\delta^{-2(n-1)}$, where $\alpha_n$ is a sufficiently large dimensional constant, then $P_r(\T)=\varnothing$, and the claimed inequality is trivially true. 
\end{enumerate}

\par Now, assume that the following estimate holds when either $\widetilde{r}\geq 2r$ or $\widetilde{\delta}\geq \delta^{-\epsilon^{10}}\delta$ or both, and $\widetilde{r}\geq c_1(\epsilon)\,\widetilde{\delta}^{2(n-1)-\epsilon}\left|\widetilde{\T}\right|$: \[\left|P_{\widetilde{r}}\left(\widetilde{\T}\right)\right|\leq c_2(\epsilon)\,\widetilde{\delta}^{-\epsilon}\widetilde{W}^{-2(n-1)}\widetilde{r}^{-2}\left|\widetilde{\T}\right|^2.\]
\par We can assume in addition and without loss of generality that $W\lesssim_{n,\epsilon}\delta^{-1+\frac{\epsilon}{2(n-1)}}$ since otherwise \[\delta^{2(n-1)-\epsilon}|\T|\gtrsim_{n}\delta^{2(n-1)-\epsilon}W^{2(n-1)}\delta^{-2(n-1)}\gtrsim_{n,\epsilon}\delta^{-2(n-1)-\epsilon}\gg\delta^{-2(n-1)},\]
and it would follow from the second base case that $P_r(\T)=\varnothing$. This assumption implies $W\ll\left(\lambda\delta\right)^{-1}$, where $\lambda:=\delta^{-\frac{\epsilon}{100n}}$. 
\par Let $P:=P_r(\T)\setminus P_{2r}(\T)$. If $|P_r(\T)|\geq 10|P|$, then by our induction hypothesis \begin{align*}
|P_r(\T)|\leq\frac{10}{9}|P_{2r}(\T)|
&\leq\frac{10}{9}c_2(\epsilon)\,\delta^{-\epsilon}W^{-2(n-1)}(2r)^{-2}|\T|^2\\
&\leq c_2(\epsilon)\,\delta^{-\epsilon}W^{-2(n-1)}r^{-2}|\T|^2.
\end{align*}
Thus, we may also assume $|P_r(\T)|<10|P|$. 

\par If we are in the thin case, then, taking $E:=r$ in \Cref{lem:heavy_ball_lemma}, we get \begin{align*}
|P|&\lesssim(-\log\delta)\delta^{-\frac{\epsilon}{25n}+2\epsilon^3}r^{-2}\delta^{-2(n-1)}|\T|\\
&\lesssim(-\log\delta)\delta^{-\frac{\epsilon}{25n}+2\epsilon^3}r^{-2}W^{-2(n-1)}|\T|^2\\
&=(-\log\delta)\delta^{\left(1-\frac{1}{25n}\right)\epsilon+2\epsilon^3}\cdot\delta^{-\epsilon}r^{-2}W^{-2(n-1)}|\T|^2,
\end{align*}
where, in the second inequality, we used the fact that $|\T|\sim W^{2(n-1)}\delta^{-2(n-1)}N$.
It follows that for $\delta$ sufficiently small, \[|P_r(\T)|<10|P|\leq c_2(\epsilon)\,\delta^{-\epsilon}W^{-2(n-1)}r^{-2}|\T|^2.\]

\par If we are in the thick case, then there exists a collection $\widetilde{P}$ of essentially distinct $2\lambda\delta$-balls $Q_j$ such that $\bigcup_jQ_j$ contains a $\gtrsim(-\log\delta)^{-1}$ fraction of the $\delta$-balls in $P$, and each $Q_j$ intersects~$\gtrsim r\delta^{-\frac{n-1}{50n}\epsilon+2n\epsilon^3}$ tubes of $\T$. Thicken each $\delta$-tube in $\T$ to a $\lambda\delta$-tube, and call the new collection $\widetilde{\T}$. For each $1\leq M\lesssim\lambda^{2(n-1)}$, let $\widetilde{\T}_M$ be the set of $2\lambda\delta$-tubes in $\widetilde{\T} $ which contain $\sim M$ tubes\footnote{At least $M$ and strictly fewer than $2M$ tubes, to be precise.} of $\T$. By dyadic pigeonholing, we can find a particular $M_0$ such that $\widetilde{\T}_{M_0}$ contains a $\gtrsim\left((n-1)\log\lambda\right)^{-1}$ fraction of the incidences between $\widetilde{\T}$ and $\widetilde{P}$. 
\par Observe that for each $2\lambda\delta$-tube $\widetilde{T}$, every $2\lambda\delta$-tube that is not essentially distinct from $\widetilde{T}$ lies completely in the tube $10\widetilde{T}$. Let $\widetilde{\T}_{M_0,\max}\subseteq\widetilde{\T}_{M_0}$ be a maximal subset of essentially distinct $2\lambda\delta$-tubes. Then, \[I(10\widetilde{\T}_{M_0,\max},\widetilde{P})\supseteq I(\widetilde{\T}_{M_0},\widetilde{P}),\]
where $10\widetilde{\T}_{M_0,\max}:=\left\{10\widetilde{T}:\widetilde{T}\in\widetilde{\T}_{M_0,\max}\right\}$.
Further, since each $\delta$-tube is contained in $\lesssim_n1$ essentially distinct $2\lambda\delta$-tubes,\[|\T|\gtrsim_n M_0\left|10\widetilde{\T}_{M_0,\max}\right|=M_0\left|\widetilde{\T}_{M_0,\max}\right|.\]
From now on we write $\dbtilde{\T}:=10\widetilde{\T}_{M_0,\max}$.

\par It's easy to show by contradiction that for $\delta$ sufficiently small, a $\gtrsim_\epsilon \delta^{\epsilon^7}$ fraction\footnote{The choice of the exponent $7$ in the expression $\delta^{\epsilon^7}$ and the exponent $5$ in the expression $\delta^{\epsilon^5}$ in the next line is somewhat arbitrary, other powers could work just as fine for our proof, as long as the latter exponent is sufficiently large and the former is even larger.} of the $2\lambda\delta$-balls $Q_j\in\widetilde{P}$ intersects $\gtrsim\delta^{\epsilon^5}\cdot r\delta^{-\frac{n-1}{50n}\epsilon+2n\epsilon^3}M_0^{-1}$ tubes of $\dbtilde{\T}$, with the implicit constant in the latter inequality being, say, $\widetilde{C}=\widetilde{C}(n,\epsilon)$. Since each $2\lambda\delta$-ball contains $\lesssim_n\lambda^{2n}$ essentially distinct $\delta$-balls, \[
|P|\lesssim(-\log\delta)\lambda^{2n}\left|\widetilde{P}\right|
\lesssim(-\log\delta)\delta^{-\frac{\epsilon}{50}}\cdot\delta^{-\epsilon^7}\left|P_{\widetilde{r}}\left(\dbtilde{\T}\right)\right|,\]
where $\widetilde{r}:=\left\lfloor\widetilde{C}r\delta^{-\frac{n-1}{50n}\epsilon+2n\epsilon^3+\epsilon^5}M_0^{-1}\right\rfloor$.
\par It follows from the given properties of $\T$ that $\dbtilde{\T}$, together with $\widetilde{\delta}:=2\lambda\delta$ and $\widetilde{W}:=W$, meets the spacing condition in the induction hypothesis. Moreover, \begin{align*}
\widetilde{r}\gtrsim r\delta^{-\frac{n-1}{50n}\epsilon+2n\epsilon^3+\epsilon^5}M_0^{-1}
&\geq c_1(\epsilon)\delta^{2(n-1)-\epsilon}|\T|M_0^{-1}\delta^{-\frac{n-1}{50n}\epsilon+2n\epsilon^3+\epsilon^5}\\
&\gtrsim_\epsilon (\lambda\delta)^{2(n-1)-\epsilon}\lambda^{-2(n-1)+\epsilon}\left|\dbtilde{\T}\right|\delta^{-\frac{n-1}{50n}\epsilon+2n\epsilon^3+\epsilon^5}\\
&\sim(\lambda\delta)^{2(n-1)-\epsilon}\left|\dbtilde{\T}\right|\delta^{-\frac{1}{100n}\epsilon^2+2n\epsilon^3+\epsilon^5},
\end{align*}
with $\delta^{-\frac{1}{100n}\epsilon^2+2n\epsilon^3+\epsilon^5}\gg 1$. Thus, by the induction hypothesis, \begin{align*}
\left|P_{\widetilde{r}}\left(\dbtilde{\T}\right)\right|&\leq c_2(\epsilon)\,(2\lambda\delta)^{-\epsilon}W^{-2(n-1)}\widetilde{r}^{-2}\left|\dbtilde{\T}\right|^2\\
&\leq c_2(\epsilon)\,\delta^{-\epsilon}W^{-2(n-1)}r^{-2}|\T|^2\cdot\lambda^{-\epsilon}\widetilde{C}^{-2}\delta^{\frac{n-1}{25n}\epsilon-4n\epsilon^3-2\epsilon^5}.
\end{align*}
Hence,
\begin{align*}
|P_r(\T)|\lesssim|P|&\lesssim(-\log\delta)\delta^{-\frac{\epsilon}{50}-\epsilon^7}\left|P_{\widetilde{r}}\left(\dbtilde{\T}\right)\right|\\
&\lesssim c_2(\epsilon)\delta^{-\epsilon}W^{-2(n-1)}r^{-2}|\T|^2\cdot\widetilde{C}^{-2}(-\log\delta)\delta^{\frac{n-2}{50n}\epsilon+\frac{1}{100n}\epsilon^2-4n\epsilon^3-2\epsilon^5-\epsilon^7},
\end{align*}
where $\widetilde{C}^{-2}(-\log\delta)\delta^{\frac{n-2}{50n}\epsilon+\frac{1}{100n}\epsilon^2-4n\epsilon^3-2\epsilon^5-\epsilon^7}\ll1$ when $n\geq 2$. This closes the induction.
\end{proof}
Below is a restatement of \Cref{thm:intro_2nd_main}: 

\begin{theorem}\label{thm:2nd_main}
Let $n=2$~or~$3$, $1\leq W\leq\delta^{-1}$, $N_0$ be a positive integer independent of $\delta$ and $W$, and $\T$ be a collection of essentially distinct complex $\delta$-tubes essentially contained in $B_{\C^n}(\boldsymbol{0},1)$. Suppose, for some maximal family of essentially distinct complex $W^{-1}$-tubes essentially contained in $B_{\C^n}(\boldsymbol{0},1)$, each of the $W^{-1}$-tubes contains exactly $N_0$ complex $\delta$-tubes $T\in\T$. Then, for richness $r\geq\max\left\{C_{\epsilon}\delta^{2(n-1)-\epsilon}|\T|,1+N_0\right\}$, \[|P_r(\T)|\lesssim_{\epsilon,n,N_0} \delta^{-\epsilon}|\T|^{\frac{n}{n-1}}r^{-\frac{n+1}{n-1}}.\]
\end{theorem}

Note that, given any collection $\T$ of essentially distinct complex $\delta$-tubes obeying the slightly looser spacing condition that each complex $W^{-1}$-tube in some fixed maximal family contains at most $N_0$ of the tubes from $\T$, we can always augment the collection $\T$ with extra $\delta$-tubes if needed to obtain a collection $\T_{\textrm{aug}}\supseteq\T$ of essentially distinct complex $\delta$-tubes such that each complex $W^{-1}$-tube in the fixed maximal family contains exactly $N_0$ of the complex $\delta$-tubes in $\T_{\textrm{aug}}$. So we have the following quick corollary, which we will use in the proof of \Cref{thm:diff_set_size}. 

\begin{corollary}\label{cor:2nd-thm}
Let $n=2$~or~$3$, $1\leq W\leq\delta^{-1}$, $N_0$ be a positive integer independent of $\delta$ and $W$, and $\T$ be a collection of essentially distinct complex $\delta$-tubes essentially contained in $B_{\C^n}(\boldsymbol{0},1)$. Suppose, for some maximal family of essentially distinct complex $W^{-1}$-tubes essentially contained in $B_{\C^n}(\boldsymbol{0},1)$, each of the $W^{-1}$-tubes contains at most $N_0$ complex $\delta$-tubes $T\in\T$. Then, for richness $r\geq\max\left\{a_nC_{\epsilon}\delta^{2(n-1)-\epsilon}N_0W^{4(n-1)},1+N_0\right\}$, where $a_n$ denotes the smallest dimensional constant that ensures $|\T|\leq a_nN_0W^{4(n-1)}$, we have \[|P_r(\T)|\lesssim_{\epsilon,n,N_0} \delta^{-\epsilon}W^{4n}r^{-\frac{n+1}{n-1}}.\]
\end{corollary}

%--------Pf of 2nd Thm--------
\begin{proof}[Proof of \Cref{thm:2nd_main}]
We will again do a double induction on the scale $\delta$ and the richness $r$. Here are the base cases: \begin{enumerate}[label=(\Roman*)]
\item Suppose $r\gtrsim_n N_0\delta^{-2n+2}$, with the implicit dimensional constant in the lower bound being sufficiently large. Then, $r\gtrsim_n N_0\delta^{-2n+2}\geq N_0W^{2n-2}$ would imply $P_r(\T)=\varnothing$, and the claimed inequality would be trivially true.

\item Suppose $\delta\sim_{n}1$. Then on one hand, 
\[\left|P_r(\T)\right|\leq\left|P_{1+N_0}(\T)\right|\lesssim_n\left(W^{2(2n-2)}\right)^{\!2}\!\cdot N_0^2\leq\delta^{-(8n-8))}N_0^2\sim_{n,N_0}1.\]
On the other hand, $P_r(\T)$ is nonempty only if $r\lesssim_n N_0W^{2n-2}$, in which case \begin{align*}
|\T|^{\frac{n}{n-1}}r^{-\frac{n+1}{n-1}}&\gtrsim_n\left(N_0W^{2(2n-2)}\right)^{\!\frac{n}{n-1}}\!\left(N_0W^{2n-2}\right)^{\!-\frac{n+1}{n-1}}\\
&\geq N_0^{-\frac{1}{n-1}}\delta^{-(2n-2)}\sim_{n,N_0}1.
\end{align*}
Thus the claimed inequality follows.
\end{enumerate}

\par For the sake of induction, let's assume that the bound below holds when either $\widetilde{r}\geq2r$ or $\widetilde{\delta}\geq\delta^{-\epsilon^{10}}\delta$ or both, and $\widetilde{r}\geq\max\left\{C_{\epsilon}\widetilde{\delta}^{2(n-1)-\epsilon}\left|\widetilde{\T}\right|,1+N_0\right\}$: 
\[\left|P_{\widetilde{r}}\left(\widetilde{\T}\right)\right|\leq c_{n,\epsilon,N_0}\, \delta^{-\epsilon}\left|\widetilde{\T}\right|^{\frac{n}{n-1}}\widetilde{r}^{-\frac{n+1}{n-1}}.\]
We will also assume without loss of generality that $W<\delta^{-1+\frac{\epsilon}{2n}}\ll\left(\lambda\delta\right)^{-1}$, with~$\lambda:=\delta^{-\frac{\epsilon}{100n}}$. 

\par By dyadic pigeonholing, we can find an $\alpha=\alpha(r)>W^{-1}$ such that the subcollection \[P_{r,\alpha}(\T):=\left\{p\in P_r(\T):\substack{\text{The maximal angle between complex tubes}\\ \text{of $\T$ passing through the $\delta$-ball $p$ is $\sim\alpha$.}}\right\}\]
has size $\gtrsim(-\log\delta)^{-1}|P_r(\T)|$. The lower bound $W^{-1}$ comes from the requirement $r\geq1+N_0$, which infers that the complex $\delta$-tubes passing through any $r$-rich $\delta$-ball~$p$ cannot all lie in the same $W^{-1}$-tube. 

\par Let's first consider two cases where $r<\delta^{-\epsilon^3}$. 
\begin{enumerate}[label=(\alph*)]
\item Suppose $\alpha\leq\delta^{\frac{3}{2}\epsilon^3}$. Let $\{\tau\}\subseteq\C\P^{n-1}$ be a maximal set of $\alpha$-separated directions. For each $\tau$, cover $B_{\C^n}(\boldsymbol{0},1)$ with $\sim_{n}\alpha^{-2(n-1)}$ complex $\alpha$-tubes $\square_\tau$ in the direction $\tau$. Let $\T_{\square_\tau}$ denote the collection of complex tubes $T\in\T$ that are essentially contained in $\square_\tau$. Then, our assumption about the spacing of the complex tubes implies that $\left|\T_{\square_\tau}\right|\lesssim_{n} N_0(\alpha W)^{4(n-1)}$. 
\par Now, fix a complex tube $\square_\tau$, and rescale $\square_\tau$ to essentially a unit ball. Then, the complex $\delta$-tubes in $\T_{\square_\tau}$ become complex $\widetilde{\delta}$-tubes, the collection of which we call $\widetilde{\T}_{\square_\tau}$. Here, $\widetilde{\delta}:=\frac{\delta}{\alpha}\geqslant\delta^{-\frac{3}{2}\epsilon^3}\delta$. This collection $\widetilde{\T}_{\square_\tau}$ might not exactly match the spacing condition in the induction hypothesis with $\widetilde{W}:=\alpha W$
as $\widetilde{\T}_{\square_\tau}$ could be sparse; but by the argument preceding \Cref{cor:2nd-thm}, we still have the following bound at the richness $\widetilde{r}:=\max\left\{\left\lceil a_nC_{\epsilon}\widetilde{\delta}^{2(n-1)-\epsilon}N_0\widetilde{W}^{4(n-1)}\right\rceil,1+N_0\right\}\leq r$: 
\begin{equation}\label{eq:induction_hypothesis_small_r_small_alpha}
\left|P_{\widetilde{r}}\left(\widetilde{\T}_{\square_\tau}'\right)\right|
\lesssim_{n,\epsilon,N_0}\widetilde{\delta}^{-\epsilon}\widetilde{W}^{4n}\widetilde{r}^{-\frac{n+1}{n-1}}
<\widetilde{\delta}^{-\epsilon}\widetilde{W}^{4n}.
\end{equation}
Rescaling the unit ball back to the complex $\alpha$-tube $\square_\tau$, we see that each $\widetilde{r}$-rich $\widetilde{\delta}$-balls becomes a complex tube of radius $\delta$ and length $\frac{\delta}{\alpha}$, which can be viewed as an almost disjoint union (or, more precisely, a finitely overlapping union) of $\sim_{n}\alpha^{-2}$ of the $\delta$-balls in $P_{\widetilde{r},\alpha}(\T)$. As a result, \begin{align*}
\left|P_{r,\alpha(r)}(\T)\right|\leq\left|P_{\widetilde{r},\alpha(r)}(\T)\right|
&\lesssim_{n}\sum_{\tau}\sum_{\square_\tau}\left|P_{\widetilde{r}}\left(\widetilde{\T}_{\square_\tau}\right)\right|\cdot\alpha^{-2}\\
&\lesssim_{n,\epsilon,N_0}\alpha^{-4(n-1)}\cdot\left(\tfrac{\delta}{\alpha}\right)^{-\epsilon}(\alpha W)^{4n}\cdot\alpha^{-2}\\
&\sim_{n,N_0}\delta^{-\epsilon}|\T|^{\frac{n}{n-1}}\cdot\alpha^{2+\epsilon},
\end{align*}
where \begin{align*}
\alpha^{2+\epsilon}\leq\delta^{\frac{3}{2}\epsilon^3\cdot(2+\epsilon)}\leq\delta^{\frac{n+1}{n-1}\epsilon^3}\!\cdot\delta^{\frac{3}{2}\epsilon^4}<r^{-\frac{n+1}{n-1}}\!\cdot\delta^{\frac{3}{2}\epsilon^4}.
\end{align*}
So, we can close the induction in the case where $r<\delta^{-\epsilon^3}$ and $\alpha\leq\delta^{\frac{3}{2}\epsilon^3}$.

\item Next, suppose $\alpha>\delta^{\frac{3}{2}\epsilon^3}$. Then by our assumption about $r$,
\[\delta^{-\epsilon}|\T|^{\frac{n}{n-1}}r^{-\frac{n+1}{n-1}}\gtrsim_{n}\delta^{-\epsilon}N_0^{\frac{n}{n-1}}W^{4n}\delta^{\frac{n+1}{n-1}\epsilon^3}.\]
Thus, when $W>\delta^{-\frac{1}{2}+\frac{\epsilon}{8}}$, we have \begin{align*}
\delta^{-\epsilon}|\T|^{\frac{n}{n-1}}r^{-\frac{n+1}{n-1}}\gtrsim_{N_0}\delta^{-2n-\frac{\epsilon}{2}+\frac{n+1}{n-1}\epsilon^3}\gg\delta^{-2n}\gtrsim|P_r(\T)|.
\end{align*}
For this reason, we may as well assume $W\leq\delta^{-\frac{1}{2}+\frac{\epsilon}{8}}$ in the rest of part~{(b)}. Then, $W^{-2}\geq\delta^{-\frac{\epsilon}{4}+1}\gg\delta$. 
Let $\widetilde{\T}$ be the collection of complex $W^{-2}$-tubes obtained by thickening each complex $\delta$-tube in $\T$, and let $\widetilde{\T}_{\max}\subseteq\widetilde{\T}$ be a maximal subcollection of essentially distinct $W^{-2}$-tubes. Since $W^{-2}< W^{-1}$, the spacing condition of $\T$ implies that $\left|\widetilde{\T}_{\max}\right|\geq N_0^{-1}\left|\widetilde{\T}\right|$. Let $Q$ be a minimal covering of $B_{\C^n}(\boldsymbol{0},1)$ consisting of $W^{-2}$-balls; and for dyadic $X,M\geq 1$, let $Q_{X,M}$ be the collection of $W^{-2}$-balls in $Q$ that contains $\sim X$ of the $\delta$-balls in $P_{r,\alpha}(\T)$ and intersects $\sim M$ of\footnote{At least $M$ and strictly less than $2M$ of the complex tubes, to be precise.} the complex $\delta$-tubes in~$\T$. We shall note that the set $Q_{X,M}$ is nonempty only if $X\lesssim\left(\frac{W^{-2}}{\delta}\right)^{\!2n}$ and $r\leq M\lesssim W^{4(n-1)}$. Then, by dyadic pigeonholing twice, we can find some particular $X_0$ and $M_0$ such that $\bigcup_{q\in Q_{X_0,M_0}}q$ covers a $\gtrsim_n(-\log\delta)^{-2}$ fraction of the $\delta$-balls in $P_{r,\alpha}(\T)$. By definition, each $\delta$-ball in $P_{r,\alpha}(\T)$ is intersected by two complex tubes in $\T$ with directions differing by $\sim\alpha$. Hence, the number of $\delta$-balls from $P_{r,\alpha}(\T)$ contained in each $q$ is $\lesssim\alpha^{-2}M_0^2$; and consequently we have \[X_0\lesssim\alpha^{-2}M_0^2\lesssim\delta^{-3\epsilon^3}M_0^2.\]
By the induction hypothesis, \[\left|P_{M_0}\left(\widetilde{\T}_{\max}\right)\right|\lesssim W^{2\epsilon}\left|\widetilde{\T}\right|^{\frac{n}{n-1}}M_0^{-\frac{n+1}{n-1}}.\]
Therefore, \begin{align*}
|P_r(\T)|&\lesssim(-\log\delta)^2\cdot X_0\cdot\left|P_{M_0}\left(\widetilde{\T}_{\max}\right)\right|\\
&\lesssim(-\log\delta)^2\delta^{-3\epsilon^3} W^{2\epsilon}\left|\widetilde{\T}_{\max}\right|^{\frac{n}{n-1}}M_0^{2-\frac{n+1}{n-1}}\\
&\leq(-\log\delta)^2\delta^{-3\epsilon^3+2\epsilon\left(-\frac{1}{2}+\frac{\epsilon}{8}\right)}|\T|^{\frac{n}{n-1}}M_0^{2-\frac{n+1}{n-1}}\\
&=(-\log\delta)^2\delta^{\frac{1}{4}\epsilon^2-3\epsilon^3}\!\cdot\delta^{-\epsilon}|\T|^{\frac{n}{n-1}}\cdot M_0^{2-\frac{n+1}{n-1}},
\end{align*}
with $M_0^{2-\frac{n+1}{n-1}}\leq 1$ when $n=2\text{ or }3$, and \[(-\log\delta)^2\delta^{\frac{1}{4}\epsilon^2-3\epsilon^3}= (-\log\delta)^2\delta^{\frac{1}{4}\epsilon^2-\left(3+\frac{n+1}{n-1}\right)\epsilon^3}\!\cdot\delta^{\frac{n+1}{n-1}\epsilon^3}\ll r^{-\frac{n+1}{n-1}}.\]
This closes the induction in the case where $r<\delta^{-\epsilon^3}$ and $\alpha>\delta^{\frac{3}{2}\epsilon^3}$.
\end{enumerate}

\par Hereinafter, suppose $r\geq\delta^{-\epsilon^3}$. If $W\leq\delta^{-\epsilon^4}$, then $P_r(\T)\neq\varnothing$ only if $r\lesssim_nN_0W^{2n-2}$, in which case \begin{align*}
\left|P_r(\T)\right|\leq\left|P_{1+N_0}(\T)\right|\lesssim_n\left(W^{2(2n-2)}\right)^{\!2}\!\cdot N_0^2
&\sim_{n} |\T|^{\frac{n}{n-1}}\left(N_0W^{2(2n-2)}\right)^{\!\frac{n-2}{n-1}}\\
&\lesssim_{n,N_0}|\T|^{\frac{n}{n-1}}r^{-\frac{n+1}{n-1}}W^{4(n-2)+2(n+1)}\\
&\leq |\T|^{\frac{n}{n-1}}r^{-\frac{n+1}{n-1}}\delta^{-(6n-6)\epsilon^4}\\
&\ll |\T|^{\frac{n}{n-1}}r^{-\frac{n+1}{n-1}}\delta^{-\epsilon}.
\end{align*}
So we will conveniently assume $W>\delta^{-\epsilon^4}$. Take $D:=\delta^{-\epsilon^4}$, and let $\Q$ be a covering of $B_{\C^n}(\boldsymbol{0},2)$ consisting of $D\delta$-balls $\q$ such that no two of the balls $\frac{1}{2}\q$ overlap\footnote{Basically, we find a packing of the set $B_{\C^n}(\boldsymbol{0},2)$ by $\frac{1}{2}D\delta$-balls and then double the radius of each ball to obtain a covering. Choosing such a covering will help us avoid ambiguity about where the chopped tubes come from later in the proof.}. By our induction hypothesis, we may assume without loss of generality that $|P|\sim|P_r(\T)|$, where $P:=P_r(\T)\setminus P_{2r}(\T)$. For any pair of intersecting~$T$ and~$\q$, the intersection $T\cap\q$ is contained in a complex tube of length $2D\delta$ and radius $\delta$, which we call $\mathring{T}$. Fix a maximal subset $\mathring{\mathbb{T}}_{\max}\subseteq\mathring{\mathbb{T}}$ consisting of essentially distinct chopped tubes that maximizes $I\!\left(P,\mathring{\mathbb{T}}_{\max}\right)$. Then in particular, \[I\!\left(P,\mathring{\mathbb{T}}_{\max}\right)\gtrsim\left(W^{-1}D\right)^{2(n-1)}I(P,\mathbb{T}).\]

\par Let $\mathring{\T}_{\q,H}$ be the set of chopped complex tubes $\mathring{T}\in\mathring{\T}_{\max}$ which are essentially contained in $\q$ and $\sim H$ of the complex tubes $T\in\T$. Because of the spacing condition, $\mathring{\T}_{\q,H}\neq\varnothing$ only if $H\lesssim (WD^{-1})^{2(n-1)}$. As in the proof of \Cref{thm:1st_main}, we first use dyadic pigeonholing to pick a particular $H_0$ such that \[\sum_{\q\in\Q}H_0\cdot\left|I\!\left(P_{\q},\mathring{\T}_{\q,H_0}\right)\right|\gtrsim_n\left(-\log\delta\right)^{-1}|I(P,\T)|,\] 
where $P_{\q}$ is the collection of $\delta$-balls $p\in P$ that are essentially contained in $\q$. 
Let $P_{\q,E}$ be the set of $\delta$-balls in $P$ which are essentially contained in $\q$ and $\sim E$ of the chopped tubes $\mathring{T}\in\mathring{\T}_{\q,H_0}$. Again, by dyadic pigeonholing, we can find a particular $E_0$ such that \[
\sum_{\q\in\Q}H_0\cdot\left|I\!\left(P_{\q,E_0},\mathring{\mathbb{T}}_{\q,H_0}\right)\right|\gtrsim_n\left(-\log\delta\right)^{-2}|I(P,\T)|.\]
By the definition of $P_{\q,E}$, we have $\left|I\!\left(P_{\q,E_0},\mathring{\mathbb{T}}_{\q,H_0}\right)\right|\sim E_0|P_{\q,E_0}|$; and by the definition of $P$, $|I(P,\T)|\sim r|P|$. It then follows from our choice of $E_0$ that \begin{equation}\label{eq:consequence_of_2nd_pigeonhole_large_r}
\sum_{\q}|P_{\q,E_0}|\gtrsim_n(-\log\delta)^{-2}\frac{r}{H_0 E_0}|P|.
\end{equation}
On the other hand, \[\sum_{\q}|P_{\q,E_0}|\leq\sum_{\q}|P_{\q}|\leq 2|P|.\]
Hence, 
\begin{equation}\label{eq:H_0E_0>r}
H_0 E_0\gtrsim|\log\delta|^{-2}r.
\end{equation}
Further, since each $\q\in P_{\q,E_0}$ is essentially contained in $\sim E_0$ of the chopped tubes $\mathring{T}\in\T_{\q,H_0}$, and each $\mathring{T}\in\T_{\q,H_0}$ is contained in $\sim H_0$ of the complex tubes $T\in\T$, each $\q$ is essentially contained in $\sim H_0 E_0$ of the complex tubes $T\in\T$. Recalling the definition $P=P_r(\T)\setminus P_{2r}(\T)$, we deduce that
\begin{equation}\label{eq:H_0E_0<r}
H_0 E_0\lesssim 2r\sim r.
\end{equation}
Thus, by \eqref{eq:consequence_of_2nd_pigeonhole_large_r}, \[|P|\lesssim_n|\log\delta|^2\sum_{\q\in\Q}\left|P_{\q,E_0}\right|\leq|\log\delta|^2\left(\sum_{\q\text{ thin}}\left|P_{\q,E_0}\right|+\sum_{\q\text{ thick}}\left|P_{\q,E_0}\right|\right).\]
Now we apply \Cref{lem:heavy_ball_lemma} to each pair of $P_{\q,E_0}$ and $\mathbullet{\T}_{\q,H_0}$, rescaled and translated so that $\q$ becomes the unit ball.
\begin{enumerate}
\item[(c)] For each $\q$ falling into the thin case, \[\left|P_{\q,E_0}\right|\lesssim D^{\frac{\epsilon}{25n}-2\epsilon^3}E_0^{-2}D^{2(n-1)}\left|\mathbullet{\T}_{\q,H_0}\right|.\]
So, if the thin case dominates, then \[|P|\lesssim_n|\log\delta|^2\sum_{\q\text{ thin}}\left|P_{\q,E_0}\right|\lesssim|\log\delta|^2\cdot D^{\frac{\epsilon}{25n}-2\epsilon^3}E_0^{-2}D^{2(n-1)}\sum_{\q}\left|\mathbullet{\T}_{\q,H_0}\right|.\]
To estimate the term $\sum_{\q}\left|\mathbullet{\T}_{\q,H_0}\right|$, let $\{\sigma\}\subseteq\C\P^{n-1}$ be a maximal set of $D^{-1}$-separated directions, and for each $\sigma$, let $\{\square_{\sigma}\}$ be a minimal covering of $B_{\C^n}(\boldsymbol{0},1)$ consisting of complex $D^{-1}$-tubes $\square_\sigma$ in the direction $\sigma$. Let $\T_{\square_\sigma}$ denote the collection of complex tubes $T\in\T$ that are essentially contained in $\square_\sigma$. Each chopped complex tube $\mathbullet{T}\in\mathbullet{\T}_{\q,H_0}$ intersects essentially $\sim H_0$ complex tubes of $\T$, all of which lie essentially in the same $\square_\sigma$. Rescale each $\square_\sigma$ to essentially a unit ball, and let $\widetilde{\T}_{\square_\sigma}$ be the resulting set of complex $\widetilde{\delta}$-tubes, where $\widetilde{\delta}:=D\delta$. Then, $\widetilde{\T}_{\square_\sigma}$ meets the spacing condition with $\widetilde{W}:=WD^{-1}$. Moreover, \[\sum_{\q}\left|\mathbullet{\T}_{\q,H_0}\right|\lesssim\sum_{\sigma}\sum_{\square_\sigma}\left|P_{\widetilde{r}}\left(\widetilde{\T}_{\square_\sigma}\right)\right|,\]
where $\widetilde{r}\sim H_0$ (or, more precisely, $\widetilde{r}$ is a constant fraction of $H_0$). Since \[\widetilde{\delta}=D\delta=\delta^{1-\epsilon^4}\gg \delta^{1-\epsilon^{10}}, \]
we can apply our induction hypothesis and get \begin{align*}
\left|P_{\widetilde{r}}\left(\widetilde{\T}_{\square_\sigma}\right)\right|\lesssim\widetilde{\delta}^{-\epsilon}\left|\widetilde{\T}_{\square_\sigma}\right|^{\frac{n}{n-1}}\widetilde{r}^{-\frac{n+1}{n-1}}&\sim(D\delta)^{-\epsilon}\left(WD^{-1}\right)^{4(n-1)\cdot\frac{n}{n-1}}H_0^{-\frac{n+1}{n-1}}\\
&\sim(D\delta)^{-\epsilon}D^{-4n}|\T|^{\frac{n}{n-1}}H_0^{-\frac{n+1}{n-1}},
\end{align*}
which implies \begin{align*}
\sum_{\q}\left|\mathbullet{\T}_{\q,H_0}\right|&\lesssim D^{4(n-1)}\cdot(D\delta)^{-\epsilon}D^{-4n}|\T|^{\frac{n}{n-1}}H_0^{-\frac{n+1}{n-1}}\\
&=D^{-4}(D\delta)^{-\epsilon}|\T|^{\frac{n}{n-1}}H_0^{-\frac{n+1}{n-1}},
\end{align*}
and consequently \begin{align*}
|P_r(\T)|&\lesssim|\log\delta|^2\cdot D^{2n-6+\frac{\epsilon}{25n}-2\epsilon^3}(D\delta)^{-\epsilon}E_0^{-2}H_0^{-\frac{n+1}{n-1}}|\T|^{\frac{n}{n-1}}\\
&\lesssim|\log\delta|^{2+2\frac{n+1}{n-1}}\delta^{-\epsilon} D^{2n-6+\left(\frac{1}{25n}-1\right)\epsilon-2\epsilon^3}E_0^{-2+\frac{n+1}{n-1}}r^{-\frac{n+1}{n-1}}|\T|^{\frac{n}{n-1}},
\end{align*}
with the second inequality following from \eqref{eq:H_0E_0>r}. When $-2+\frac{n+1}{n-1}\geq0$, i.e. when $n=2\text{ or }3$, \[E_0^{-2+\frac{n+1}{n-1}}\lesssim \left(D^{2(n-1)}\right)^{-2+\frac{n+1}{n-1}},\]
and hence \begin{align*}
|P_r(\T)|&\lesssim|\log\delta|^{\frac{4n}{n-1}} D^{\left(\frac{1}{25n}-1\right)\epsilon-2\epsilon^3}\cdot\delta^{-\epsilon}r^{-\frac{n+1}{n-1}}|\T|^{\frac{n}{n-1}},
\end{align*}
where $|\log\delta|^{\frac{4n}{n-1}} D^{\left(\frac{1}{25n}-1\right)\epsilon-2\epsilon^3}\ll 1$.

\item[(d)] For each $\q$ falling into the thick case, there exists a collection $\widetilde{P}_{\q}$ of finitely overlapping balls $\widetilde{p}$ of width $\sim D^{\frac{\epsilon}{100n}}\delta$ such that the union $\bigcup_{\widetilde{p}\in\widetilde{P}_{\q}}\widetilde{p}$ contains a $\gtrsim (\log D)^{-1}$ fraction of the balls $p\in P_{\q,E_0}$, and each ball $\widetilde{p}$ intersects $\gtrsim E_0D^{\frac{n-1}{50n}\epsilon-2n\epsilon^3}$ of the complex tubes $\mathbullet{T}\in\mathbullet{\T}_{\q,H_0}$. Set \[\widetilde{P}:=\bigcup_{\q\text{ thick}}\widetilde{P}_{\q}.\]
Then, each $\widetilde{p}\in\widetilde{P}$ intersects $\gtrsim N_0E_0 D^{\frac{n-1}{50n}\epsilon-2n\epsilon^3}$ of the complex tubes $T\in\T$. To apply the induction hypothesis, we fatten each complex $\delta$-tube $T\in\T$ into a complex tube $\widetilde{T}$ of width $\sim D^{\frac{\epsilon}{100n}}\delta$, and call the collection of all these complex tubes $\widetilde{\T}$. Then, $\widetilde{P}\subseteq P_{\widetilde{r}}(\widetilde{\T})$ for $\widetilde{r}$ a sufficiently small constant fraction of $H_0E_0 D^{\frac{n-1}{50n}\epsilon-2n\epsilon^3}$, and for such $\widetilde{r}$ we have \begin{align*}
\widetilde{r}\sim H_0E_0 D^{\frac{n-1}{50n}\epsilon-2n\epsilon^3}
&\gtrsim(\log\delta)^{-2}rD^{\frac{n-1}{50n}\epsilon-2n\epsilon^3}\\
&\gg r \\
&\geq \max\left\{C_{\epsilon}\widetilde{\delta}^{2(n-1)-\epsilon}\left|\widetilde{\T}\right|,1+N_0\right\}.
\end{align*}
The collection $\widetilde{\T}$ meets the spacing condition with $\widetilde{\delta}:= C D^{1+\frac{\epsilon}{100n}}\delta$ and $\widetilde{W}:=W$, where \[\widetilde{W}=W\lesssim\delta^{-\frac{1}{2}-\frac{1}{4(n-1)}+\frac{\epsilon}{4(n-1)}}\ll\delta^{-1+\epsilon^4\left(1+\frac{\epsilon}{100n}\right)}\sim\widetilde{\delta}^{-1}.\]
Hence, \begin{align*}
\sum_{\q\text{ thick}}\left|\widetilde{P}_{\q}\right|\leq \left|P_{\widetilde{r}}\left(\widetilde{\T}\right)\right|&\lesssim\widetilde{\delta}^{-\epsilon}\widetilde{r}^{-\frac{n+1}{n-1}}\left|\widetilde{\T}\right|^{\frac{n}{n-1}}\\
&\lesssim \delta^{\epsilon-\epsilon^5\left(1+\frac{\epsilon}{100n}\right)}|\log\delta|^{\frac{2(n+1)}{n-1}}\cdot\delta^{-\epsilon}r^{-\frac{n+1}{n-1}}|\T|^{\frac{n}{n-1}},
\end{align*}
and \begin{align*}
|P_r(\T)|\lesssim (\log \delta)^2\sum_{\q\text{ thick}}\left|P_{\q,E_0}\right|
&\lesssim (\log \delta)^2(\log D)\sum_{\q\text{ thick}}\left|\widetilde{P}_{\q}\right|\\
&\lesssim \delta^{\epsilon-\epsilon^5\left(1+\frac{\epsilon}{100n}\right)}\left|\log\delta\right|^{3+\frac{2(n+1)}{n-1}}\cdot\delta^{-\epsilon}r^{-\frac{n+1}{n-1}}|\T|^{\frac{n}{n-1}}\\
&\ll \delta^{-\epsilon}r^{-\frac{n+1}{n-1}}|\T|^{\frac{n}{n-1}}.
\end{align*}
\end{enumerate}
The induction is now complete.
\end{proof}
\section{Application to a variant of Falconer's distance set problem}\label{sect_app}

In this section, we investigate the following question:
\begin{question}\label{quest:diff_set_size}
Fix $2<s<4$. Let $E$ be a collection of $\sim N\delta^{-s}$ many\footnote{say $\frac{1}{2}\left|B_{\C^2}(\vec{0},1)\right|N\delta^{-s}<|E|<2\left|B_{\C^2}(\vec{0},1)\right|N\delta^{-s}$} $\delta$-balls in $B_{\C^2}(\vec{0},1)$ with at most $N$ $\delta$-balls in each ball of radius $\delta^{s/4}$. How many finitely overlapping $\delta$-balls are needed to cover the difference set \[\Delta(E):=\left\{(x_1-x_2)^2+(y_1-y_2)^2\in\C:\Vec{p}_1=(x_1,y_1),\Vec{p}_2=(x_2,y_2)\in\bigcup E\right\}?\]
\end{question}
Our work will build up to a proof of \Cref{cor:intro_diff_set_size}, which we have restated below as \Cref{thm:diff_set_size} for convenience.
\begin{theorem}
\label{thm:diff_set_size}
The number of distinct $\delta$-balls needed to cover the difference set \[\Delta(E):=\left\{(x_1-x_2)^2+(y_1-y_2)^2\in\C:\Vec{p}_1=(x_1,y_1),\Vec{p}_2=(x_2,y_2)\in\bigcup E\right\}\]
is $\gtrsim_{\epsilon,s,N}\delta^{-2+\epsilon}$ for all $\epsilon>0$. 
\end{theorem}

We shall emphasize that our complex analogue of the distance set is the set of squared sums of differences, rather than the set of squared sums of the norms of the differences.
\begin{definition}
Given any pair of points $\Vec{p}_1=(x_1,y_1)$ and $\Vec{p}_2=(x_2,y_2)$ in $\C^2$, define \[\Delta\left(\Vec{p}_1,\Vec{p}_2\right):=(x_1-x_2)^2+(y_1-y_2)^2.\]
\end{definition}

To build a connection between the difference set problem and the incidence estimates, we introduce an auxiliary line for each pair of points in $\C^2$.

\begin{definition}
Given any pair of points $\Vec{p}_1=(x_1,y_1)$ and $\Vec{p}_2=(x_2,y_2)$ in $\C^2$, define the auxiliary line
\[l_{\Vec{p}_1,\Vec{p}_2}:=\left\{\left(\tfrac{x_1+x_2}{2}+\tfrac{y_1-y_2}{2}z,\tfrac{y_1+y_2}{2}-\tfrac{x_1-x_2}{2}z,z\right):z\in\C\right\}.
\]
To simplify future notation, let's write \[\vv_{\Vec{p}_1,\Vec{p}_2}:=\left(\tfrac{y_1-y_2}{2},-\tfrac{x_1-x_2}{2},1\right).\]
Then, \[l_{\Vec{p}_1,\Vec{p}_2}= \left \{ \left(\tfrac{\Vec{p}_1+\Vec{p}_2}{2},0\right)+z\vv_{\Vec{p}_1,\Vec{p}_2}: z \in \C \right \}.
\] 
\end{definition}

The intuition behind this definition comes from the Elekes-Sharir framework used in the distinct distance problem in $\R^2$ (see e.g. \cite{ErdosDistinctDist_Real}): Over $\R$ instead of $\C$, the orthogonal projection of $l_{\Vec{p}_1,\Vec{p}_2}$ onto the first two coordinates is the perpendicular bisector of the segment between the points $\Vec{p}_1$ and $\Vec{p}_2$ in $\R^2$, and the line $l_{\Vec{p}_1,\Vec{p}_2}$ itself parametrizes all the rigid motions mapping $\Vec{p}_1$ to $\Vec{p}_2$. Two things to note: \begin{enumerate}
\item For any four points $\Vec{p}_j=(x_j,y_j)\in\C^2$, $j=1,2,3,4$, the auxiliary lines $l_{\Vec{p}_1,\Vec{p}_3}$ and $l_{\Vec{p}_2,\Vec{p}_4}$ intersect if and only if \[\Delta\left(\Vec{p}_1,\Vec{p}_2\right) = \Delta\left(\Vec{p}_3,\Vec{p}_4\right).\]
We will prove a more quantitative version of this statement in \Cref{prop:quantitative_auxline_intersection}.
\item For any two distinct points $\Vec{p}_1, \Vec{p}_2\in\C^2$, the auxiliary lines $l_{\Vec{p}_1,\Vec{p}_2}$ and $l_{\Vec{p}_2,\Vec{p}_1}$ are different.
\end{enumerate}

\begin{definition}
\label{blackboardT}
Given a collection $E$ of $\delta$-balls as described in \Cref{quest:diff_set_size}, we can choose two balls $B'$ and $B''$ of radius $C_1\leq 0.01$ with centres at least $C_2\geq 1.2$ apart such that each of the two subcollections $E':=\{q\in E:q\subseteq B'\}$ and $E'':=\{q\in E:q\subseteq B''\}$ contains a $\gtrsim1$ fraction of the balls in $E$. 
\par Define $\T$ to be the collection of complex almost $\delta$-tubes \[\left\{T_{q_1,q_2}:(q_1,q_2)\in (E'\times E'')\cup(E''\times E')\right\},\]
where \[T_{q_1,q_2}:=\left(\bigcup_{(\Vec{p}_1,\Vec{p}_2)\in q_1\times q_2}l_{\Vec{p}_1,\Vec{p}_2}\right)\cap B_{\C^3}(\boldsymbol{0},1).\]
\end{definition}

%----START of justification----
The following proposition justifies our choice to call the auxiliary objects $T_{q_1,q_2}$ ``almost" tubes. 

\begin{proposition}
Let $q_1$ and $q_2$ be any two $\delta$-balls contained in $B_{\C^2}(\vec{0},1)$, and let $\Vec{\sigma}_1=\left(x_{\Vec{\sigma}_1},y_{\Vec{\sigma}_1}\right)$ and $\Vec{\sigma}_2=\left(x_{\Vec{\sigma}_2},y_{\Vec{\sigma}_2}\right)$ denote the centres of $q_1$ and $q_2$, respectively. Then, \begin{enumerate}
\item\label{itm:almost_tube_contained_in_larger_tube}$T_{q_1,q_2}\subseteq \nb_{2\delta}\left(l_{\Vec{\sigma}_1,\Vec{\sigma}_2}\right)\cap B_{\C^3}(\boldsymbol{0},1)$; and
\item\label{itm:almost_tube_contains_smaller_tube} $T_{q_1,q_2}\supseteq \nb_{\delta/2}\left(l_{\Vec{\sigma}_1,\Vec{\sigma}_2}\right)\cap B_{\C^3}(\boldsymbol{0},1)$.
\end{enumerate}
\end{proposition}

The proof of this proposition involves nothing more than basic geometry and algebra, and is included only for the sake of completeness. Readers may trust their instinct and skip to the next page.

\begin{proof}
To prove \cref{itm:almost_tube_contained_in_larger_tube}, consider any two points $\Vec{p}_1=(x_1,y_1)\in q_1$ and $\Vec{p}_2=(x_2,y_2)\in q_2$, and any scalar $z\in\C$ such that \[\left(\frac{x_1+x_2}{2}+z\frac{y_1-y_2}{2},\frac{y_1+y_2}{2}-z\frac{x_1-x_2}{2},z\right)\in B_{\C^3}(\boldsymbol{0},1).\]
Then clearly $|z|<1$. And triangle inequality gives \begin{align*}
&\left\|\left(\frac{x_1+x_2}{2}+z\frac{y_1-y_2}{2},\frac{y_1+y_2}{2}-z\frac{x_1-x_2}{2},z\right)\right.\\
&\qquad\qquad\left.-\left(\frac{x_{\Vec{\sigma}_1}+x_{\Vec{\sigma}_2}}{2}+z\frac{y_{\Vec{\sigma}_1}-y_{\Vec{\sigma}_2}}{2},\frac{y_{\Vec{\sigma}_1}+y_{\Vec{\sigma}_2}}{2}-z\frac{x_{\Vec{\sigma}_1}-x_{\Vec{\sigma}_2}}{2},z\right)\right\|\\
&\leq\frac{\left\|\Vec{p}_1+\Vec{p_2}-\Vec{\sigma}_1-\Vec{\sigma}_2\right\|}{2}+|z|\frac{\left\|\Vec{p}_1-\Vec{p_2}-\Vec{\sigma}_1+\Vec{\sigma}_2\right\|}{2}\\
&\leq\frac{\left\|\Vec{p}_1-\Vec{\sigma}_1\right\|+\left\|\Vec{p}_2-\Vec{\sigma}_2\right\|}{2}+\frac{\left\|\Vec{p}_1-\Vec{\sigma}_1\right\|+\left\|\Vec{p}_2-\Vec{\sigma}_2\right\|}{2}\\
&<2\delta.
\end{align*}
\par To prove \cref{itm:almost_tube_contains_smaller_tube}, consider any scalar $z\in\C$ and any point $(a,b,c)\in B_{\C^3}\left(\boldsymbol{0},\frac{\delta}{2}\right)$ such that \[\left(\frac{x_{\Vec{\sigma}_1}+x_{\Vec{\sigma}_2}}{2}+z\frac{y_{\Vec{\sigma}_1}-y_{\Vec{\sigma}_2}}{2}+a,\frac{y_{\Vec{\sigma}_1}+y_{\Vec{\sigma}_2}}{2}-z\frac{x_{\Vec{\sigma}_1}-x_{\Vec{\sigma}_2}}{2}+b,z+c\right)\in B_{\C^2}(\boldsymbol{0},1).\]
We would like to solve for the following system under the additional requirement $\left(r_1,s_1\right),\left(r_2,s_2\right)\in B_{\C^2}\left(\Vec{0},\delta\right)$:
\[\begin{dcases}
\frac{\left(x_{\Vec{\sigma}_1}+r_1\right)+\left(x_{\Vec{\sigma}_2}+r_2\right)}{2}+(z+c)\frac{\left(y_{\Vec{\sigma}_1}+s_1\right)+\left(y_{\Vec{\sigma}_2}+s_2\right)}{2}&\\
&\hspace{-8em}= \frac{x_{\Vec{\sigma}_1}+x_{\Vec{\sigma}_2}}{2}+z\frac{y_{\Vec{\sigma}_1}-y_{\Vec{\sigma}_2}}{2}+a\\
\frac{\left(y_{\Vec{\sigma}_1}+s_1\right)+\left(y_{\Vec{\sigma}_2}+s_2\right)}{2}+(z+c)\frac{\left(x_{\Vec{\sigma}_1}+r_1\right)+\left(x_{\Vec{\sigma}_2}+r_2\right)}{2}&\\
&\hspace{-8em}= \frac{x_{\Vec{\sigma}_1}+x_{\Vec{\sigma}_2}}{2}+z\frac{y_{\Vec{\sigma}_1}-y_{\Vec{\sigma}_2}}{2}+b
\end{dcases}.\]
After rearrangement, the above system is equivalent to \[\begin{dcases}
\frac{r_1+r_2}{2}+(z+c)\frac{s_1-s_2}{2}=a-c\frac{y_{\Vec{\sigma}_1}-y_{\Vec{\sigma}_2}}{2}\\
\frac{s_1+s_2}{2}+(z+c)\frac{r_1-r_2}{2}=a-c\frac{x_{\Vec{\sigma}_1}-x_{\Vec{\sigma}_2}}{2}
\end{dcases}.\]
Forcing $r_1=r_2$ and $s_1=s_2$, we find a solution \[\begin{dcases}
r_1=r_2=a-c\frac{y_{\Vec{\sigma}_1}-y_{\Vec{\sigma}_2}}{2}\\
s_1=s_2=b+c\frac{x_{\Vec{\sigma}_1}-x_{\Vec{\sigma}_2}}{2}
\end{dcases},\]
which has the property that \begin{align*}
\left\|\left(r_1,s_1\right)\right\|=\left\|\left(r_2,s_2\right)\right\|&=\left\|(a,b)+c\left(-\frac{y_{\Vec{\sigma}_1}-y_{\Vec{\sigma}_2}}{2},\frac{x_{\Vec{\sigma}_1}-x_{\Vec{\sigma}_2}}{2}\right)\right\|\\
&\leq\|(a,b)\|+|c|\frac{\left\|\Vec{\sigma}_1-\Vec{\sigma}_2\right\|}{2}\\
&<\frac{\delta}{2}+\frac{\delta}{2}\cdot 1\\
&=\delta \qedhere
\end{align*}
\end{proof}
%----END of justification----

Next, define $Q$ to be the set of all quadruples $(q_1,q_2,q_3,q_4)\in E'\times E''\times E''\times E'$ such that \[\delta>\dist\left(\Delta(q_1,q_2),\Delta(q_3,q_4)\right):=\min_{(\Vec{p}_1,\Vec{p}_2,\Vec{p}_3,\Vec{p}_4)\in q_1\times q_2\times q_3\times q_4}\left|\Delta(\Vec{p}_1,\Vec{p}_2 )-\Delta(\Vec{p}_3,\Vec{p}_4)\right|.\]
The following two proposition will then allow us to resolve \Cref{quest:diff_set_size}.

\begin{proposition}
\label{prop:quantitative_auxline_intersection}
If a quadruple of $\delta$-balls $(q_1,q_2,q_3,q_4)\in E'\times E''\times E''\times E'$ is in $Q$, then the corresponding complex $\delta$-tubes $T_{q_1,q_3}$ and $T_{q_2,q_4}$ intersect in a $\delta$-ball.
\end{proposition}

\begin{proposition}
\label{prop:spacing_check}
The collection $\T$ meets the spacing condition of \Cref{cor:2nd-thm}, with $W \sim \delta^{-s/4}$ and $N_0=N^2$. 
\end{proposition}

\begin{remark}
Here, we set  $W \sim \delta^{-s/4}$ (rather than $W = \delta^{-s/4}$) to account for various implied constants that appear when proving that a spacing condition for the $\delta$-balls results in a spacing condition for the associated tubes.  Ultimately, we want to make $W^{-1}$ a small positive multiple of $\delta^{s/4}$ (e.g. by $1/1000$); thus we should think of $W$ as a multiple of $\delta^{-s/4}$ by a constant $> 1$.  One could alternatively read the proposition as the statement, ``there exists $C > 1$ so that if $W = C \delta^{-s/4}$ and $N_0 = N^2$, then the collection $\T$ meets the spacing condition of \Cref{cor:2nd-thm}."
\end{remark}

Below, we give a proof of \Cref{thm:diff_set_size} assuming both \Cref{prop:quantitative_auxline_intersection} and \Cref{prop:spacing_check}.  We later prove each of these propositions in a separate subsection.

\begin{proof}[Proof of \Cref{thm:diff_set_size}]
\par Assuming \Cref{prop:quantitative_auxline_intersection}, and counting the incidences, we have \[|Q|\lesssim\sum_{r\geq 2\text{ dyadic}}r^2\left|P_{r}(\T)\right|.\]
Assuming \Cref{prop:spacing_check}, we know from \Cref{thm:2nd_main} that for $r\geq\delta^{4-\epsilon}|\T|$, \[|P_r(\T)|\lesssim\delta^{-\epsilon}|\T|^{\frac{3}{2}}r^{-2}.\]
The same bound holds true for $2\leq r<\delta^{4-\epsilon}|\T|$ since $|\T|\sim\delta^{-2s}$ and there are at most $\sim\delta^{-6}$ essentially distinct $\delta$-balls in the unit ball $B_{\C^3}(\boldsymbol{0},1)$. Hence, \[|Q|\lesssim\left|\log_2\delta^{4}\right|\cdot\delta^{-\epsilon}|\T|^{\frac{3}{2}}\sim|\log\delta|\cdot\delta^{-\epsilon-3s}.\]
Finally, by the Cauchy-Schwarz inequality, the number of $\delta$-balls required to cover the difference set $\Delta(E',E'')\subseteq\Delta(E)$ has the following lower bound on it: \[\#\Delta(E',E'')\geq\frac{(\# E'\cdot\# E'')^2}{|Q|}\gtrsim|\log\delta|^{-1}\delta^{-s+\epsilon}\geq\delta^{-2+\epsilon}.\]
So we have proved $\Cref{thm:diff_set_size}$.
\end{proof}

\subsection{Proof of \Cref{prop:quantitative_auxline_intersection}}

Fix a quadruple of $\delta$-balls $q_1\times q_2\times q_3\times q_4\in Q$ and a quadruple of points $(\vec{p}_1,\vec{p}_2,\vec{p}_3,\vec{p}_4)\in q_1\times q_2\times q_3\times q_4$.

\begin{claim}
Given any $\vec{p}_1,\vec{p}_4\in B'$ and $\vec{p}_2,\vec{p}_3\in B''$, the angle between $l_{\vec{p}_1,\vec{p}_3}$ and $l_{\vec{p}_2,\vec{p}_4}$ is $\gtrsim 1$.
\end{claim}

\begin{proof}
We recall from \Cref{sect_def} the usual embedding $\iota:\C^n\to\R^{2n}$ splitting each complex number into its real and imaginary part, and note that $\iota$ commutes with addition and multiplication by real scalars. By the variational definition of the angle, all we need to show is that the angle between the two vectors $\iota\left(\vv_{\vec{p}_1,\vec{p}_3}\right)$ and $\iota\left(z\vv_{\vec{p}_2,\vec{p}_4}\right)$, which we denote by $\theta_z$, is $\gtrsim 1$ for all $z\in\C$ with $|z|=1$. The cosine law gives \begin{align*}
\cos\theta_z&=\frac{\left\|\iota\left(\vv_{\vec{p}_1,\vec{p}_3}\right)\right\|^2+\left\|\iota\left(z\vv_{\vec{p}_2,\vec{p}_4}\right)\right\|^2-\left\|\iota\left(\vv_{\vec{p}_1,\vec{p}_3}-z\vv_{\vec{p}_2,\vec{p}_4}\right)\right\|^2}{2\left\|\iota\left(\vv_{\vec{p}_1,\vec{p}_3}\right)\right\|\left\|\iota\left(z\vv_{\vec{p}_2,\vec{p}_4}\right)\right\|}\\
&=\frac{\left\|\vv_{\vec{p}_1,\vec{p}_3}\right\|^2+\left\|\vv_{\vec{p}_2,\vec{p}_4}\right\|^2-\left\|\vv_{\vec{p}_1,\vec{p}_3}-z\vv_{\vec{p}_2,\vec{p}_4}\right\|^2}{2\left\|\vv_{\vec{p}_1,\vec{p}_3}\right\|\left\|\vv_{\vec{p}_2,\vec{p}_4}\right\|}\\
&=\frac{\left(\left\|\frac{\vec{p}_1-\vec{p}_3}{2}\right\|^2+1\right)+\left(\left\|\frac{\vec{p}_2-\vec{p}_4}{2}\right\|^2+1\right)-\left(\left\|\frac{\vec{p}_1-\vec{p}_3}{2}-z\frac{\vec{p}_2-\vec{p}_4}{2}\right\|^2+|1-z|^2\right)}{2\sqrt{\left\|\frac{\vec{p}_1-\vec{p}_3}{2}\right\|^2+1}\sqrt{\left\|\frac{\vec{p}_2-\vec{p}_4}{2}\right\|^2+1}}\\
&<\frac{2\left(\left(\frac{C_2+2C_1}{2}\right)^2+1\right)-\left(\left\|\frac{\vec{p}_1-\vec{p}_3}{2}-z\frac{\vec{p}_2-\vec{p}_4}{2}\right\|^2+|1-z|^2\right)}{2\left(\left(\frac{C_2-2C_1}{2}\right)^2+1\right)}\\
&=1+\frac{4C_1C_2-\left(\left\|\frac{\vec{p}_1-\vec{p}_3}{2}-z\frac{\vec{p}_2-\vec{p}_4}{2}\right\|^2+|1-z|^2\right)}{2\left(\left(\frac{C_2-2C_1}{2}\right)^2+1\right)},
\end{align*}
so we just need to show that for some constant $\epsilon_1>0$,\[\frac{4C_1C_2-\left(\left\|\frac{\vec{p}_1-\vec{p}_3}{2}-z\frac{\vec{p}_2-\vec{p}_4}{2}\right\|^2+|1-z|^2\right)}{2\left(\left(\frac{C_2-2C_1}{2}\right)^2+1\right)}<-\epsilon_1,\]
or in other words, \[\left\|\frac{\vec{p}_1-\vec{p}_3}{2}-z\frac{\vec{p}_2-\vec{p}_4}{2}\right\|^2+|1-z|^2>4C_1C_2+2\epsilon_1\left(\left(\frac{C_2-2C_1}{2}\right)^2+1\right).\]
Since $C_2<2$ and $C_1\leq 0.01$, it suffices to show for some $\epsilon_1>0$\[\left\|\frac{\vec{p}_1-\vec{p}_3}{2}-z\frac{\vec{p}_2-\vec{p}_4}{2}\right\|^2+|1-z|^2>0.08+4\epsilon_1.\]
Let's assume to the contrary that there exists some $z\in\C$ with $|z|=1$ such that \[\left\|\frac{\vec{p}_1-\vec{p}_3}{2}-z\frac{\vec{p}_2-\vec{p}_4}{2}\right\|^2+|1-z|^2\leq 0.09.\]
Then, on one hand, 
\[|1-z|\leq 0.3;\]
while on the other hand, \begin{align*}
0.3&\geq\left\|\frac{\vec{p}_1-\vec{p}_3}{2}-z\frac{\vec{p}_2-\vec{p}_4}{2}\right\|\\
&=\left\|(1+z)\frac{\vec{p}_1-\vec{p}_3}{2}+z\frac{\left(\vec{p}_4-\vec{p}_1\right)-\left(\vec{p}_2-\vec{p}_3\right)}{2}\right\|\\
&\geq |1+z|\left\|\frac{\vec{p}_1-\vec{p}_3}{2}\right\|-\left\|\frac{\left(\vec{p}_4-\vec{p}_1\right)-\left(\vec{p}_2-\vec{p}_3\right)}{2}\right\|\\
&> \frac{C_2-2C_1}{2}|1+z|-2C_1\\
&\geq \frac{1}{2}|1+z|-0.02,
\end{align*}
which, after rearrangement, gives \[
0.64>|1+z|=|z-(-1)|.\]
But then by the triangle inequality,
\[2=|1-(-1)|\leq|1-z|+|z-(-1)|<0.3+0.64=0.94,\]
which is contradictory. 
\par To summarize, in order for $\theta_z$ to be small, the first two coordinates of $\vv_{\vec{p}_1,\vec{p}_3}$ and of $\vv_{\vec{p}_2,\vec{p}_4}$ force $z$ to be closer to $-1$ while the last coordinate of $\vv_{\vec{p}_1,\vec{p}_3}$ and of $\vv_{\vec{p}_2,\vec{p}_4}$ force $z$ to be closer to $1$; yet the two cannot be achieved at the same time; thus $\theta_z$ cannot be too small. In conclusion, \[\angle\left(l_{\vec{p}_1,\vec{p}_3},l_{\vec{p}_2,\vec{p}_4}\right)=\min_{z\in\C:|z|=1}\theta_z\geq\arccos\left(1-\frac{0.09-0.08}{4}\right)=\arccos\left(\frac{399}{400}\right).\hfill\qedhere\]
\end{proof}

Proving \Cref{prop:quantitative_auxline_intersection} now reduces to proving the following claim:

\begin{claim}
\label{deltaBallIntersection}
The distance between $l_{\vec{p}_1,\vec{p}_3}$ and $l_{\vec{p}_2,\vec{p}_4}$ is $<\delta/2$. Consequently, the complex tubes $T_{q_1,q_3}$ and $T_{q_2,q_4}$ intersect in (essentially) a $\delta$-ball in $B_{\C_n}\left(\boldsymbol{0},1\right)$
\end{claim}

Let's consider the mapping $A:\C^2\to l_{\vec{p}_1,\vec{p}_3}-l_{\vec{p}_2,\vec{p}_4}\subseteq\C^3$ given by \[\boldsymbol{A}\left(t,t'\right):=\left(\tfrac{\vec{p}_1+\vec{p}_3}{2}-\tfrac{\vec{p}_2+\vec{p}_4}{2},0\right)+t\vv_{\vec{p}_1,\vec{p}_3}-t'\vv_{\vec{p}_2,\vec{p}_4}.\]
Observe that \[\left|\left(\vec{p}_1-\vec{p}_3\right)-\left(\vec{p}_2-\vec{p}_4\right)\right|=\left|\left(\vec{p}_1+\vec{p}_4\right)-\left(\vec{p}_2+\vec{p}_3\right)\right|\geq 2C_2-4C_1.\]
If $\left|x_1-x_3-x_2+x_4\right|\geq\left|y_1-y_3-y_2+y_4\right|$, then $\left|x_1-x_3-x_2+x_4\right|\geq C_2-2C_1$, and, evaluating $\|\boldsymbol{A}\|$ at $t=t'=\frac{y_1+y_3-y_2-y_4}{x_1-x_3-x_2+x_4}$, we find that \[\left\|\boldsymbol{A}\left(t,t'\right)\right\|^2=\frac{\left|\Delta(\vec{p}_1,\vec{p}_2 )-\Delta(\vec{p}_3,\vec{p}_4)\right|^2}{\left(2|x_1-x_3-x_2+x_4|\right)^2}<\left(\frac{\delta}{2C_2-4C_1}\right)^{\!2}<\left(\frac{\delta}{2}\right)^{\!2},\]
with $|t|=\left|t'\right|\leq\frac{4C_1}{C_2-2C_1}<\frac{1}{25}$. If $\left|x_1-x_3-x_2+x_4\right|\leq\left|y_1-y_3-y_2+y_4 \right|$, then, evaluating $\|\boldsymbol{A}\|$ at $t=t'=\frac{x_1+x_3-x_2-x_4}{y_1-y_3-y_2+y_4}$, we again have \[\left\|\boldsymbol{A}\left(t,t'\right)\right\|<\frac{\delta}{2},\]
with $|t|=\left|t'\right|<\frac{1}{25}$. Thus, in either case, there exists a ball of radius $\frac{\delta}{4}$ which contains a point from $l_{\vec{p}_1,\vec{p}_3}$ and a point from $l_{\vec{p}_2,\vec{p}_4}$; moreover, this $\frac{\delta}{4}$-ball lies completely in the ball $B_{\C^3}\!\left(\boldsymbol{0},1\right)$ because
\[\left\|\tfrac{\vec{p}_1+\vec{p}_3}{2}\right\|,\left\|\tfrac{\vec{p}_2+\vec{p}_4}{2}\right\|<\sqrt{1-\left(\frac{C_2-2C_1}{2}\right)^2}<\frac{\sqrt{3}}{2},\]
and \[\left(\tfrac{\vec{p}_1+\vec{p}_3}{2},0\right)+tv_{\vec{p}_1,\vec{p}_3},\left(\tfrac{\vec{p}_2+\vec{p}_4}{2},0\right)+t'v_{\vec{p}_2,\vec{p}_4}<\frac{\sqrt{3}}{2}+\frac{1}{25}\cdot\sqrt{\left(\frac{2}{2}\right)^2+1}<0.93.\]
The $\delta$-ball with the same centre witnesses the incidence between the complex tubes $T_{q_1,q_3}$ and $T_{q_2,q_4}$. \qed

\subsection{Proof of \Cref{prop:spacing_check}}
Let $W^{-1} = c \delta^{s/4}$ for some small constant $c \in (0,1)$ to be determined.
Let $R$ be a complex $W^{-1}$-tube in $B_{\C^3}(\boldsymbol{0},1)$. We want to show that if $c$ is sufficiently small, then $R$ contains no more than $N^2$-many tubes of $\T$.  To accomplish this, we will partition $\T$ into at most $N^2$ parts and show that $R$ contains at most $1$ tube from each of these parts. To partition $\T$, we first partition $E$.  

We can think of partitioning $E$ as assigning a color to each $\delta$-ball in $E$.  We adapt a graph-theoretic argument that we first encountered in \cite{Fu_Ren_Incidences}, in the proof of their Proposition 4.1. We build a graph whose vertex set is the set of centres of the $\delta$-balls in $E$ and place edges so that two vertices have an edge between them if and only if the corresponding $\delta$-balls are both contained in a larger $\delta^{s/4}$-ball.  The resulting graph has maximum degree $N-1$, so by Brooks' Theorem, the chromatic number of the graph is at most $N$; that is, we can color the $\delta$-balls of $E$ using $N$ colors in a way that guarantees that no two $\delta$-balls of the same color are in any $\delta^{s/4}$-ball. 

For $j = 1, \dots, N$, let $E_j$ be the set of $\delta$-balls which were colored with color $j$. For $1 \leq j,k \leq N$, let $\T_{j,k} \subset \T$ be the collection 
\[
\T_{j,k} = \{ T_{q_1, q_2} \in \T : q_1 \in E_j \text{ and } q_2 \in E_k \}.
\]

Since the $N^2$ subsets $\{ \T_{j,k} \}_{1 \leq j,k \leq N}$ partition $\T$, proving \Cref{prop:spacing_check} now reduces to proving the following claim:
\begin{claim}
\label{claim:one_tube_atmost}
For any pair of indices $j,k$ with $1 \leq j,k \leq N$, a $W^{-1}$-tube $R$ contains at most $1$ tube in $\T_{j,k}$. 
\end{claim}

We now prove the claim.
\begin{proof}
Fix a pair of indices $(j, k)$. We will consider two distinct  $\delta$-tubes in $\T_{j,k}$ and show that there is no $W^{-1}$-tube in $B_{\C^3}(\boldsymbol{0},1)$ that essentially contains both of these $\delta$-tubes, provided that $W^{-1} = c \delta^{s/4}$ for a sufficiently small constant $c$. Although the proof is notation-heavy and may appear complicated, it ultimately reduces to an application of the parallelogram law.

Let $\vec{\sigma}_1 = (x_{\vec{\sigma}_1}, y_{\vec{\sigma}_1})$ and $\vec{\sigma}_2 = (x_{\vec{\sigma}_2}, y_{\vec{\sigma}_2})$ be the respective centres of $\delta$-balls $q_1$ and $q_2$ in $E_j$, and let $\vec{\sigma}_3 = (x_{\vec{\sigma}_3}, y_{\vec{\sigma}_3})$ and $\vec{\sigma}_4 = (x_{\vec{\sigma}_4}, y_{\vec{\sigma}_4})$ be the centres of $q_3$ and $q_4$ in $E_k$. We want to show that there is no $W^{-1}$-tube that essentially contains both of the tubes\footnote{One should note that, despite the similarity in notation, the tubes $T_{q_1, q_3}$ and $T_{q_2, q_4}$ do \emph{not} necessarily satisfy the same hypotheses as the tubes appearing in \Cref{prop:quantitative_auxline_intersection}.  In particular, in 
\Cref{prop:quantitative_auxline_intersection}, we assume that $(q_1, q_2, q_3, q_4) \in E' \times E'' \times E'' \times E'$.  In the context of our present proof, we merely know that the pair $(q_1, q_3)$ is in one of $E' \times E''$ or $E'' \times E'$, but we do not know which.  Similarly, we know that the pair $(q_2, q_4)$ is in one of $E' \times E''$ or $E'' \times E'$, but we do not know which.} $T_{q_1, q_3}$ and $T_{q_2, q_4}$.  To accomplish this, we embed $T_{q_1, q_3}$ and $T_{q_2, q_4}$ into $\R^6$ and parameterize their respective images by functions $\boldsymbol{G}_{1,3}, \boldsymbol{G}_{2,4}: B_{\R^2}\left(\vec{0}, \frac{1}{2} \right) \to \R^6$. We will show that for all $(\alpha, \beta)$ in  a set consisting of at least half the area of $B_{\R^2}(\vec{0}, \frac{1}{2})$, the distance from the point $\boldsymbol{G}_{1,3}(\alpha, \beta)$ to the image of $\boldsymbol{G}_{2,4}$ exceeds $W^{-1}$.  

Since $q_1$ and $q_2$ are both in $E_j$, it follows that $\| \vec{\sigma}_1 - \vec{\sigma}_2 \| \geq \delta^{s/4}$.  Similarly, since $q_3$ and $q_4$ are both in $E_k$, it follows that $\| \vec{\sigma}_3 - \vec{\sigma}_4 \| \geq \delta^{s/4}$.
To leverage this separation, we will apply 
the parallelogram law to the vectors $\| \vec{\sigma}_1 - \vec{\sigma}_2\|$ and $\| \vec{\sigma}_3 - \vec{\sigma}_4\|$ to get
\begin{equation}
\label{parallelogram law}
\begin{split}
&2 \| \vec{\sigma}_1 - \vec{\sigma}_2 \|^2 + 2 \| \vec{\sigma}_3 - \vec{\sigma}_4 \|^2\\
&\quad= \| (\vec{\sigma}_1 - \vec{\sigma}_2) + (\vec{\sigma}_3 - \vec{\sigma}_4)\|^2 + \| (\vec{\sigma}_1 - \vec{\sigma}_2) - (\vec{\sigma}_3 - \vec{\sigma}_4)\|^2 
\\
&\quad= \| (\vec{\sigma}_1 + \vec{\sigma}_3) - (\vec{\sigma}_2 + \vec{\sigma}_4)\|^2 + \| (\vec{\sigma}_1 - \vec{\sigma}_3) - (\vec{\sigma}_2 - \vec{\sigma}_4)\|^2.
\end{split}
\end{equation}
We will ultimately take cases according to which of the terms on the right-hand side of (\ref{parallelogram law}) is larger.  In either case, the dominant term will be comparable to $2 \| \vec{\sigma}_1 - \vec{\sigma}_2 \|^2 + 2 \| \vec{\sigma}_3 - \vec{\sigma}_4 \|^2$, which is, in turn, at least $4 (\delta^{s/4})^2$.

The tube $T_{q_1, q_3}$ is roughly the $\delta$-neighborhood of the complex line segment
\[
\overline {l_{\vec{\sigma}_1,\vec{\sigma}_3}} =\left\{ \left(\tfrac{x_{\vec{\sigma}_1}+x_{\vec{\sigma}_3}}{2},\tfrac{y_{\vec{\sigma}_1}+y_{\vec{\sigma}_3}}{2},0 \right) + z \left( \tfrac{y_{\vec{\sigma}_1}-y_{\vec{\sigma}_3}}{2}, -\tfrac{x_{\vec{\sigma}_1}-x_{\vec{\sigma}_3}}{2}, 1 \right):z\in\C, |z| \leq \tfrac{1}{2} \right\},
\]
and $T_{q_2, q_4}$ is the $\delta$-neighborhood of the complex line segment
\[
\overline{l_{\vec{\sigma}_2,\vec{\sigma}_4}} =\left\{ \left(\tfrac{x_{\vec{\sigma}_2}+x_{\vec{\sigma}_4}}{2},\tfrac{y_{\vec{\sigma}_2}+y_{\vec{\sigma}_4}}{2}, 0 \right) + z \left( \tfrac{y_{\vec{\sigma}_2}-y_{\vec{\sigma}_4}}{2}, -\tfrac{x_{\vec{\sigma}_2}-x_{\vec{\sigma}_4}}{2}, 1 \right): z\in\C, |z| \leq \tfrac{1}{2} \right\}.
\]

The complex lines $l_{\vec{\sigma}_1,\vec{\sigma}_3}$ and $l_{\vec{\sigma}_2,\vec{\sigma}_4}$ correspond\footnote{The correspondence is via the canonical embedding $\iota:\C^3 \to \R^6$.  The real vectors $\left( \tfrac{c_1 - c_3}{2}, \tfrac{d_1- d_3}{2}, - \tfrac{(a_1 - a_3)}{2}, - \tfrac{(b_1-b_3)}{2}, 1, 0 \right)$ and $\left( -\tfrac{(d_1- d_3)}{2}, \tfrac{ c_1 - c_3}{2}, \tfrac{b_1-b_3}{2}, - \tfrac{(a_1 - a_3)}{2}, 0, 1 \right)$ are the respective images of the complex vectors $(\vec{u}_1, 1)$ and $(i \vec{u_1}, i)$ under $\iota$. } to real two-planes $P_{1,3}$ and $P_{2,4}$. If we write $x_{\vec{\sigma}_j} = a_j + b_j i$ and $y_{\vec{\sigma}_j} = c_j + d_j i$ with $a_j,b_j,c_j,d_j\in\R$ for $j = 1, \dots, 4$, then we can define parameterizations $\boldsymbol{G}_{1,3}, \boldsymbol{G}_{2,4}: \R^2 \to \R^6$ by 
\[
\begin{split}
\boldsymbol{G}_{1,3}(\alpha, \beta) &= \left( \tfrac{a_1 + a_3}{2}, \tfrac{b_1 + b_3}{2}, \tfrac{c_1 + c_3}{2}, \tfrac{d_1 + d_3}{2}, 0, 0 \right) \\
&\quad+ \alpha \left( \tfrac{c_1 - c_3}{2}, \tfrac{d_1- d_3}{2}, - \tfrac{(a_1 - a_3)}{2}, - \tfrac{(b_1-b_3)}{2}, 1, 0 \right) \\
&\quad+ \beta \left( -\tfrac{(d_1- d_3)}{2}, \tfrac{ c_1 - c_3}{2}, \tfrac{b_1-b_3}{2}, - \tfrac{(a_1 - a_3)}{2}, 0, 1 \right)
\end{split}
\]
and
\[
\begin{split}
\boldsymbol{G}_{2,4}(\alpha, \beta) &= \left( \tfrac{a_2 + a_4}{2}, \tfrac{b_2 + b_4}{2}, \tfrac{c_2 + c_4}{2}, \tfrac{d_2 + d_4}{2}, 0, 0 \right) \\
&\quad+ \alpha \left( \tfrac{c_2 - c_4}{2}, \tfrac{d_2- d_4}{2}, - \tfrac{(a_2 - a_4)}{2}, - \tfrac{(b_2-b_4)}{2}, 1, 0 \right) \\
&\quad+ \beta \left( -\tfrac{(d_2- d_4)}{2}, \tfrac{ c_2 - c_4}{2}, \tfrac{b_2-b_4}{2}, - \tfrac{(a_2 - a_4)}{2}, 0, 1 \right).
\end{split}
\]

The complex line segments $\overline {l_{\vec{\sigma}_1,\vec{\sigma}_3}}$ and $\overline{l_{\vec{\sigma}_2,\vec{\sigma}_4}}$ correspond to the respective images under $\boldsymbol{G}_{1,3}$ and $\boldsymbol{G}_{2,4}$ of the disk $B_{\R^2}\left(\vec{0},\tfrac{1}{2}\right) = \left \{ (\alpha, \beta) \in \R^2: (\alpha^2 + \beta^2)^{1/2} \leq \tfrac{1}{2} \right \}$.
We will consider how the distance
$d(\boldsymbol{G}_{1,3}(\alpha, \beta), P_{2,4})$
varies as $(\alpha, \beta)$ ranges over the ball $B_{\R^2} \left(\vec{0}, \frac{1}{2} \right)$.  To show that no $W^{-1}$-box (essentially) contains both $T_{q_1,q_3}$ and $T_{q_2,q_4}$, it is sufficient to show that the inequality
\begin{equation}
\label{suffCondition1}
d(\boldsymbol{G}_{1,3}(\alpha, \beta), P_{2,4}) > W^{-1} + 4 \delta
\end{equation}
holds for all $(\alpha, \beta)$ in a set that accounts for
at least half 
of the area of $B_{\R^2} \left(\vec{0}, \tfrac{1}{2} \right)$.
(Here, 
\[
d(\boldsymbol{G}_{1,3}(\alpha, \beta), P_{2,4}) = \min_{(\alpha', \beta') \in \R^2} \| \boldsymbol{G}_{1,3}(\alpha, \beta) - \boldsymbol{G}_{2,4}(\alpha', \beta') \|,
\]
i.e. $d(\boldsymbol{G}_{1,3}(\alpha, \beta$ represents the perpendicular distance from the point $\boldsymbol{G}_{1,3}(\alpha, \beta)$ to the plane $P_{2,4}$.)
No matter our eventual choice of $c$, we will have $W^{-1} > 4 \delta$ for $\delta$ sufficiently small, so 
(\ref{suffCondition1}) reduces to showing that
\begin{equation}
\label{suffCondition2}
d(\boldsymbol{G}_{1,3}(\alpha, \beta), P_{2,4}) > 2W^{-1}
\end{equation}
holds for all $(\alpha, \beta)$ in a set that accounts for
at least half 
of the area of $B_{\R^2} \left(\vec{0}, \tfrac{1}{2} \right)$.

By \Cref{claim:perpendicular distance vs distance between corresponding points} ---  which is stated and proved below --- the perpendicular distance from $\boldsymbol{G}_{1,3}(\alpha, \beta)$ to the plane $P_{2,4}$ is comparable to the distance from $\boldsymbol{G}_{1,3}(\alpha, \beta)$ to $\boldsymbol{G}_{2,4}(\alpha, \beta)$; specifically, 
\begin{equation}
\label{impliedConstant}
d(\boldsymbol{G}_{1,3}(\alpha, \beta), P_{2,4}) \geq \frac{1}{\sqrt{3}}
\| \boldsymbol{G}_{1,3}(\alpha, \beta) - \boldsymbol{G}_{2,4}(\alpha, \beta)\|,
\end{equation}
so it suffices to show that
\begin{equation}
\label{suffCondition3}
\| \boldsymbol{G}_{1,3}(\alpha, \beta)- \boldsymbol{G}_{2,4}(\alpha, \beta)\| > 2 \sqrt{3} W^{-1}
\end{equation}
holds for all $(\alpha, \beta)$ in a set that accounts for at least half 
of the area of $B_{\R^2} \left(\vec{0}, \tfrac{1}{2} \right)$.

To estimate $\| \boldsymbol{G}_{1,3}(\alpha ,\beta) - \boldsymbol{G}_{2,4}(\alpha, \beta)\|$, we begin by writing 
\[
\boldsymbol{G}_{1,3}(\alpha, \beta) - \boldsymbol{G}_{2,4}(\alpha, \beta) = \boldsymbol{u} + \alpha \boldsymbol{v} + \beta \boldsymbol{w},
\]
where
\[
\begin{split}
\boldsymbol{u} = \left(\tfrac{(a_1+a_3) - (a_2+a_4)}{2}, \tfrac{(b_1+b_3) - (b_2+b_4)}{2}, \tfrac{(c_1+c_3) - (c_2+c_4)}{2}, \tfrac{(d_1+d_3) - (d_2+d_4)}{2}, 0, 0\right),
\end{split}
\]
\[
\boldsymbol{v} = \left(\tfrac{(c_1-c_3) - (c_2-c_4)}{2}, \tfrac{(d_1-d_3) - (d_2-d_4)}{2}, \tfrac{-(a_1-a_3) + (a_2-a_4)}{2}, \tfrac{-(b_1-b_3) + (b_2-b_4)}{2}, 0, 0 \right),
\]
and
\[
\boldsymbol{w} =  \left(\tfrac{-(d_1-d_3) + (d_2-d_4)}{2}, \tfrac{(c_1-c_3) - (c_2-c_4)}{2}, \tfrac{(b_1-b_3) - (b_2-b_4)}{2}, \tfrac{-(a_1-a_3) + (a_2-a_4)}{2}, 0, 0 \right).
\]
We note that
\[
\| \boldsymbol{u} \| =  \frac{1}{2}\| (\vec{\sigma}_1 + \vec{\sigma}_3) - (\vec{\sigma}_2 + \vec{\sigma}_4)\|.
\]
Meanwhile, the vectors $\boldsymbol{v}$ and $\boldsymbol{w}$ are orthogonal to each other and have the same magnitude; specifically, 
\[
\| \boldsymbol{v}\| = \| \boldsymbol{w}\| = \frac{1}{2}\| (\vec{\sigma}_1 - \vec{\sigma}_3) - (\vec{\sigma}_2 - \vec{\sigma}_4)\|.
\]
This motivates us to take cases on the relative sizes of $\left\| (\vec{\sigma}_1 + \vec{\sigma}_3) - (\vec{\sigma}_2 + \vec{\sigma}_4)\right\|$ and $\left\| (\vec{\sigma}_1 - \vec{\sigma}_3) - (\vec{\sigma}_2 - \vec{\sigma}_4)\right\|$; specifically, we take cases according to whether 
\[
\left\| (\vec{\sigma}_1 - \vec{\sigma}_3) - (\vec{\sigma}_2 - \vec{\sigma}_4)\right\| \leq 2\sqrt{2} \left\| (\vec{\sigma}_1 + \vec{\sigma}_3) - (\vec{\sigma}_2 + \vec{\sigma}_4)\right\|.
\]
(Here, the constant $2 \sqrt{2}$ is chosen to ensure that the right-hand side of a later inequality - namely, inequality (\ref{posMult}) -  is a positive multiple of $\left\| (\vec{\sigma}_1 - \vec{\sigma}_3) - (\vec{\sigma}_2 - \vec{\sigma}_4)\right\|$.) \\

\noindent \textbf{Case 1}:
First, suppose that
\begin{equation}
\label{CaseDivider1}
\left\| (\vec{\sigma}_1 - \vec{\sigma}_3) - (\vec{\sigma}_2 - \vec{\sigma}_4)\right\| \leq 2 \sqrt{2} \left\| (\vec{\sigma}_1 + \vec{\sigma}_3) - (\vec{\sigma}_2 + \vec{\sigma}_4)\right\|.
\end{equation}
By (\ref{parallelogram law}) and (\ref{CaseDivider1}), we have that 
\[
4 (\delta^{s/4})^2 \leq \left(1 + (2 \sqrt{2})^2 \right) \left\| (\vec{\sigma}_1 + \vec{\sigma}_3) - (\vec{\sigma}_2 + \vec{\sigma}_4)\right\|^2 = 9 \left\| (\vec{\sigma}_1 + \vec{\sigma}_3) - (\vec{\sigma}_2 + \vec{\sigma}_4)\right\|^2.
\]
We will consider the function 
\[
F(\alpha, \beta) := \| \boldsymbol{G}_{1,3}(\alpha, \beta) - \boldsymbol{G}_{2,4}(\alpha, \beta) \|^2
\]
and will show that, provided we take $W^{-1}$ to be a sufficiently small multiple of $\delta^{s/4}$, we have
\[
F(\alpha, \beta) > 12 W^{-2}
\]
for all $(\alpha, \beta)$ in a set consisting of at least half of $B_{\R^2} \left(\vec{0}, \tfrac{1}{2} \right)$.
First, we note that 
\[
F(0,0) = \| \boldsymbol{u} \|^2 = \frac{1}{4} \left\| (\vec{\sigma}_1 + \vec{\sigma}_3) - (\vec{\sigma}_2 + \vec{\sigma}_4)\right\|^2 = \frac{1}{4} \left( \frac{4}{9}^{\,} \delta^{s/2} \right) = \frac{1}{9}^{\,}\delta^{s/2}
\]
It is natural to consider when we have
\[
F(\alpha, \beta) \geq F(0,0).
\]
To this end, note that $F$ is a quadratic function of $\alpha$ and $\beta$.  Specifically, using the fact that $\boldsymbol{v} \cdot \boldsymbol{w} = 0$, we compute that
\begin{equation}
\label{FFormula}
\begin{split}
F(\alpha, \beta) &= ( \boldsymbol{u} + \alpha \boldsymbol{v} + \beta \boldsymbol{w} ) \cdot  ( \boldsymbol{u} + \alpha \boldsymbol{v} + \beta \boldsymbol{w} ) \\
& = \| \boldsymbol{v} \|^2 \alpha^2 +  \| \boldsymbol{w} \|^2  \beta^2 + (\boldsymbol{v} \cdot \boldsymbol{u}) \alpha + (\boldsymbol{w} \cdot \boldsymbol{u}) \beta + \|  \boldsymbol{u} \|^2. 
\end{split}
\end{equation}
From this formula, it is apparent that we have $F(\alpha, \beta) > F(0,0)$ for all points $(\alpha, \beta)$ in a half-space.  There are a couple of ways to see this, which we explain below:
\begin{enumerate}
\item We observe that if $F(\alpha, \beta) < F(0,0)$, then we must have $(\boldsymbol{v} \cdot \boldsymbol{u}) \alpha + (\boldsymbol{w} \cdot \boldsymbol{u}) \beta < 0$, which implies that the antipodal point $(-\alpha, -\beta)$ satisfies $F(-\alpha, -\beta) > F(0,0)$.
\item We observe that the graph of $F$ over $\R^2$ is a paraboloid.  If $(0,0)$ is the vertex of the paraboloid, then we have that $F(\alpha, \beta) \geq F(0,0)$ for all $(\alpha, \beta)$.  Otherwise, if the vertex is not at $(0,0)$, let $C$ be the unique level curve of $F$ passing through $(0,0)$.  Draw a line $\ell$ through $(0,0)$ tangent to $C$.  Then for all points $(\alpha, \beta)$ lying on the opposite side of $\ell$ as the vertex, we must have $F(\alpha, \beta) \geq F(0,0)$. (Note that the collection of points $(\alpha, \beta)$ lying on the opposite side of $\ell$ as the vertex consists of half of the disk  $B_{\R^2} \left( \vec{0}, \frac{1}{2} \right)$.)
\end{enumerate}
Either argument shows that for all $(\alpha, \beta)$ in a region consisting of at least half of $B_{\R^2} \left( \vec{0}, \frac{1}{2} \right)$, we have
\[
F(\alpha, \beta) \geq F(0,0) \geq \frac{1}{9}^{\,} \delta^{s/2},
\]
which means that
\[
\| \boldsymbol{G}_{1,3}(\alpha, \beta) - \boldsymbol{G}_{2,4}(\alpha, \beta) \| \geq \frac{1}{3}^{\,} \delta^{s/4}
\]
\noindent \textbf{Case 2}:
Now, suppose that 
\begin{equation}
\label{CaseDivider2}
\| (\vec{\sigma}_1 - \vec{\sigma}_3) - (\vec{\sigma}_2 - \vec{\sigma}_4)\| \geq 2\sqrt{2} \| (\vec{\sigma}_1 + \vec{\sigma}_3) - (\vec{\sigma}_2 + \vec{\sigma}_4)\|.
\end{equation}
In this case, since $\boldsymbol{v} \cdot \boldsymbol{w} = 0$ and $\| \boldsymbol{v}\| = \| \boldsymbol{w}\| = \frac{1}{2} \| (\vec{\sigma}_1 - \vec{\sigma}_3) - (\vec{\sigma}_2 - \vec{\sigma}_4)\|$, we have by the triangle inequality that 
\[
\begin{split}
\| \boldsymbol{G}_{1,3}(\alpha, \beta) - \boldsymbol{G}_{2,4}(\alpha, \beta)\| 
&\geq \frac{|(\alpha, \beta)|}{2} \| (\vec{\sigma}_1 - \vec{\sigma}_3) - (\vec{\sigma}_2 - \vec{\sigma}_4)\| \\
& \indent - \frac{1}{2}\| (\vec{\sigma}_1 + \vec{\sigma}_3) - (\vec{\sigma}_2 + \vec{\sigma}_4)\| \\
&\geq \frac{|(\alpha, \beta)|}{2} \| (\vec{\sigma}_1 - \vec{\sigma}_3) - (\vec{\sigma}_2 - \vec{\sigma}_4)\| \\
& \indent - \frac{1}{2} \left(\frac{1}{2 \sqrt{2}} \| (\vec{\sigma}_1 - \vec{\sigma}_3) - (\vec{\sigma}_2 - \vec{\sigma}_4)\| \right).
\end{split}
\]
It follows that if $|\alpha|^2 + |\beta|^2 \geq \frac{1}{2}$, then 
\begin{equation}
\label{posMult}
\begin{split}
\|\boldsymbol{G}_{1,3}(\alpha, \beta) - \boldsymbol{G}_{2,4}(\alpha, \beta)\| &\geq
\frac{1}{2} \left(\frac{1}{\sqrt{2}} - \frac{1}{2 \sqrt{2}} \right) 
\left\| (\vec{\sigma}_1 - \vec{\sigma}_3) - (\vec{\sigma}_2 - \vec{\sigma}_4)\right\|
\\
&= \frac{1}{4 \sqrt{2}} \left\| (\vec{\sigma}_1 - \vec{\sigma}_3) - (\vec{\sigma}_2 - \vec{\sigma}_4)\right\|.
\end{split}
\end{equation}
By (\ref{parallelogram law}) and (\ref{CaseDivider2}), we then have that 
\[
\begin{split}
4 (\delta^{s/4})^2 &\leq \left( \frac{1}{8} + 1 \right) \left\| (\vec{\sigma}_1 - \vec{\sigma}_3) - (\vec{\sigma}_2 - \vec{\sigma}_4)\right\|^2 = \frac{9}{8}^{\,}   \left\| (\vec{\sigma}_1 - \vec{\sigma}_3) - (\vec{\sigma}_2 - \vec{\sigma}_4)\right\|^2. \\
\end{split}
\]
Consequently, if $|\alpha|^2 + |\beta|^2 \geq \frac{1}{2}$, then
\begin{equation}
\label{bound2}
\| \boldsymbol{G}_{1,3}(\alpha,\beta) - \boldsymbol{G}_{2,4}(\alpha,\beta)\| \geq \frac{1}{4 \sqrt{2}} \left( \frac{4 \sqrt{2}}{3} \right) \delta^{s/4} = \frac{1}{3}^{\,} \delta^{s/4}.
\end{equation}
\, \, \\ \\
\noindent \textbf{Conclusion}: 
Combining Cases 1-2, we see that if we take $W^{-1} = c \delta^{s/4}$ for some $c < \frac{1}{6 \sqrt{3}}$, then our desired inequality (\ref{suffCondition3})
holds for all $(\alpha, \beta)$ in a set that accounts for at least half 
of the area of $B_{\R^2} \left(\vec{0}, \tfrac{1}{2} \right)$.
\end{proof}

\begin{claim}
\label{claim:perpendicular distance vs distance between corresponding points}
Let the maps $\boldsymbol{G}_{1,3}$ and $\boldsymbol{G}_{2,4}$ be defined as above.  Then for any point $(\alpha, \beta) \in B_{\R^2} \left(\vec{0}, \frac{1}{2} \right)$, we have that 
\begin{equation}
\label{impliedConst}
d(\boldsymbol{G}_{1,3}(\alpha, \beta), P_{2,4}) \geq \frac{1}{\sqrt{3}}
\| \boldsymbol{G}_{1,3}(\alpha, \beta) - \boldsymbol{G}_{2,4}(\alpha, \beta)\|.
\end{equation}
\end{claim}

\begin{proof}
We consider the subspace $V = \Span(\boldsymbol{e}_1, \boldsymbol{e}_2, \boldsymbol{e}_3, \boldsymbol{e}_4) \subset \R^6$.  Let $\theta_{1,3}$ denote the principal angle between $V$ and $P_{1,3}$.  Similarly, let $\theta_{2,4}$ denote the principal angle between $V$ and $P_{2,4}$.  We note that $\theta_{1,3}, \theta_{2,4} \gtrsim 1$; in particular,
\begin{equation}
\theta_{1,3}, \theta_{2,4} \geq \arctan \left( \frac{1}{\sqrt{2}} \right).
\end{equation}
To see why this is true, 
let 
\[
\vec{y}_1 = \left( \tfrac{y_{\vec{\sigma}_1}-y_{\vec{\sigma}_3}}{2},-\tfrac{x_{\vec{\sigma}_1}-x_{\vec{\sigma}_3}}{2}\right)\quad\text{and}\quad
\vec{y}_2 = \left( \tfrac{y_{\vec{\sigma}_2}-y_{\vec{\sigma}_4}}{2},-\tfrac{x_{\vec{\sigma}_2}-x_{\vec{\sigma}_4}}{2}\right)
\]
so that $\left(\vec{y}_1, 1\right)$ has the same direction as the complex line 
\[
l_{\vec{\sigma}_1,\vec{\sigma}_3} =\left\{\left(\tfrac{x_{\vec{\sigma}_1}+x_{\vec{\sigma}_3}}{2},\tfrac{y_{\vec{\sigma}_1}+y_{\vec{\sigma}_3}}{2}, 0 \right) + z \left( \tfrac{y_{\vec{\sigma}_1}-y_{\vec{\sigma}_3}}{2}, -\tfrac{x_{\vec{\sigma}_1}-x_{\vec{\sigma}_3}}{2}, 1 \right) :z\in\C\right\},
\]
and $\left(\vec{y}_2, 1\right)$ has the same direction as the complex line 
\[
l_{\vec{\sigma}_2,\vec{\sigma}_4} = \left\{\left(\tfrac{x_{\vec{\sigma}_2}+x_{\vec{\sigma}_4}}{2},\tfrac{y_{\vec{\sigma}_2}+y_{\vec{\sigma}_4}}{2}, 0 \right) + z \left( \tfrac{y_{\vec{\sigma}_2}-y_{\vec{\sigma}_4}}{2}, -\tfrac{x_{\vec{\sigma}_2}-x_{\vec{\sigma}_4}}{2}, 1 \right) :z\in\C\right\}.
\]
Then
$\| \vec{y}_1 \|, \| \vec{y}_2 \| \lesssim 1$; 
in particular, since we assume that the points \\
$(x_{\vec{\sigma}_1}, y_{\vec{\sigma}_1}), \dots, (x_{\vec{\sigma}_4}, y_{\vec{\sigma}_4})$ are in $B_{\C^2}\!\left(\vec{0},1\right)$, it follows that $\| \vec{y}_1 \|, \| \vec{y}_2 \| \leq \sqrt{2}$.  Here we show the computation for $\| \vec{y}_1 \|$:
\[
\| \vec{y}_1 \|^2 = \frac{1}{4} \left( |y_{\vec{\sigma}_1} - y_{\vec{\sigma}_3}|^2 + |x_{\vec{\sigma}_1} - x_{\vec{\sigma}_3}|^2 \right) \leq \frac{1}{4} \left( 2^2 + 2^2 \right) \leq 2.
\]

We note that each of $P_{1,3}$ and $P_{2,4}$ intersects any translate of $V$ in a single point. This is because a basis for either, along with a basis for $V$, forms a basis for $\R^6$.

After fixing $(\alpha, \beta) \in B_{\R^2}\left(\vec{0},\tfrac{1}{2}\right)$, we consider the affine subspace \[V_{\alpha, \beta} := V + (0, 0, 0, 0, \alpha, \beta).\]
To prove (\ref{impliedConst}), we consider the right triangle $\triangle ABC$, where $A$ is the point \\ $\boldsymbol{G}_{1,3}(\alpha, \beta)$, $B$  is the point $\boldsymbol{G}_{2,4}(\alpha, \beta)$, and $C$ is the point in $P_{2,4}$ that attains the minimum distance from $\boldsymbol{G}_{1,3}(\alpha, \beta)$.
Then we have that 
\[
\begin{split}
d(\boldsymbol{G}_{1,3}(\alpha, \beta), P_{2,4}) &= AB \\
&= AC \sin( \angle{ABC}) \\
&\geq AC \sin( \theta_{2,4})  \\
&\geq \sin \left( \arctan \frac{1}{\sqrt{2}} \right) AC \\
& = \frac{1}{\sqrt{3}} \, AC \\
& =\frac{1}{\sqrt{3}}  \| \boldsymbol{G}_{1,3}(\alpha, \beta) - \boldsymbol{G}_{2,4}(\alpha, \beta)\|
\end{split}
\]
(Note that to go from the fourth to the fifth line, we have used the identity $\sin( \arctan(x)) = \frac{x}{\sqrt{1 + x^2}}$.)
\end{proof}

\section{Appendix}
We return to the problem, introduced in Section 2, of estimating the volume of the intersection between the $\delta$-neighborhoods of two complex lines.

\begin{proposition}
Suppose that $\ell_1, \ell_2 \subset \C^2$ are two complex lines through the origin that make angle $\theta > 0$ to each other.   If we let $\iota: \C^2 \to \R^4$ denote the embedding that sends $(z_1,z_2)$ to $\left(\Re(z_1),\Im(z_1),\Re(z_2),\Im(z_2)\right)$, then 
\begin{equation}
\label{C2 case}
|\iota(N_{\delta}(\ell_1)) \cap \iota( (N_{\delta}(\ell_2)))| \sim \frac{\delta^4}{\sin^2 \theta}.
\end{equation}
More generally, if $\ell_1, 
\ell_2$ are lines through the origin in $\C^n$ which make angle $\theta > 0$ to each other, then 
\begin{equation}
\label{Cn case}
|\iota(N_{\delta}(\ell_1)) \cap \iota( N_{\delta}(\ell_2))| \sim \frac{\delta^{2n}}{\sin^2 \theta}.
\end{equation}
\end{proposition}

\begin{proof}
Let $\boldsymbol{u}_1$ and $\boldsymbol{u}_2$ be direction vectors for the tubes' respective central axes.  Suppose without loss of generality that $\boldsymbol{u}_1 = (1,0)$.  If $\boldsymbol{u}_2$ is parallel to $(0,1)$, then $\theta = \tfrac{\pi}{2}$, and $|\iota(N_{\delta}(\ell_1)) \cap \iota( N_{\delta}(\ell_2))| = \delta^4$.
Henceforth, we may assume that $\boldsymbol{u}_2 = \tfrac{1}{\sqrt{1+ r^2}} (1, r e^{i \zeta})$ for some $\zeta \in [0, 2 \pi)$ and some real number $r > 0$.  Note that in this case, we have that
\[
\sin^2 \theta = \frac{r^2}{1 + r^2}.
\]

Now, we set some notation.  We define 
\[
\begin{split}
\boldsymbol{v}_1 &= \iota(\boldsymbol{u}_1) = (1,0,0,0); \\
\boldsymbol{v}_2 &=   \iota(i \boldsymbol{u}_1) = (0, 1, 0, 0); \\
\boldsymbol{w}_1 &= \iota(\boldsymbol{u}_2) = \tfrac{1}{\sqrt{1+ r^2}}(1,0,r \cos \zeta , r \sin \zeta); \\
\boldsymbol{w}_2 &=   \iota(i \boldsymbol{u}_1) = \tfrac{1}{\sqrt{1+ r^2}} (0, 1, - r \sin \zeta, r \cos \zeta). \\
\end{split}
\]
We let $V = \Span\left\{ \boldsymbol{v}_1, \boldsymbol{v}_2\right\}$ and $W = \Span\left\{ \boldsymbol{w}_1, \boldsymbol{w}_2\right\}$ so that $\iota(N_{\delta}(\ell_1)) = N_{\delta}(V)$, and $\iota( N_{\delta}(\ell_2)) = N_{\delta}(W)$. We write a point in $\R^4$ as $(\vec{x}, \vec{y}) \in \R^2 \times \R^2$. Then, $N_{\delta}(V) = \{ (\vec{x}, \vec{y}): \| y\| \leq \delta \}$, and
\begin{equation}
\label{intersection volume 1}
\begin{split}
|\iota(T_1) \cap \iota( T_2)| &= |N_{\delta}(V) \cap N_{\delta}(W)| = \int_{\R^4} 1_{N_{\delta}(V)} 1_{N_{\delta}(W)} \, dV  \\
&= \int_{B_{\R^2} \left(\vec{0}, \delta \right)} \left( \int_{\- \infty}^{\infty} \int_{-\infty}^{\infty} 1_{N_{\delta}(W)}(\vec{x}, \vec{y}) \, dA_{\vec{x}} \right) \, dA_{\vec{y}}.
\end{split}
\end{equation}

\par By Lemma \ref{slice size}, we have $(\vec{x}, \vec{y}) \in N_{\delta}(W)$ if and only if $\vec{x}$ is in the ball of radius $\tfrac{\delta \sqrt{1+r^2}}{r}$ around the point $ \tfrac{1}{r} R_{\zeta}(\vec{y}) = \tfrac{1}{r}(y_1 \cos \zeta + y_2 \sin \zeta, -y_1 \sin \zeta + y_2 \cos \zeta)$, where $R_{\zeta}$ denotes the rotation around the origin through an angle $\zeta$ in the plane.  Thus, resuming from \eqref{intersection volume 1}, we conclude that
\[
\begin{split}
|\iota(T_1) \cap \iota( T_2)| &= \int_{B_{\R^2} \left(\vec{0}, \delta \right)} \left( 
\int_{B_{\R^2} \left(\tfrac{1}{r} R_{\zeta}( \vec{y}), \tfrac{\delta \sqrt{1+r^2}}{r}
\right)} 
1 \, dA_{\vec{x}} 
\right) \, dA_{\vec{y}}
 \\
&\sim \int_{B_{\R^2}\left( \vec{0}, \delta \right)} \frac{\delta^2(1 + r^2)}{r^2} \,  dA_{\vec{y}} \,  \sim \,  \frac{\delta^{4}(1+r^2)}{r^2} \,  = \,  \frac{\delta^4}{\sin^2 \theta}.
\end{split}
\]
This completes the proof of \eqref{C2 case}.  For $n>2$, the orthogonal complements of $\iota(\ell_1)$ and $\iota(\ell_1)$ intersect in a subspace of dimension $2n-4$.  Any slice of $\iota(N_{\delta}(\ell_1) \cap \iota(N_{\delta}(\ell_1))$ by a translation of this subspace is a figure like that of the $C^2$ case.  We can integrate the $4$-dimensional cross section volume over the common short directions to arrive at \eqref{Cn case}.
\end{proof}

\begin{lemma}
\label{slice size}
Let $\boldsymbol{v}_1 = \iota(\boldsymbol{u}_1) = (1,0,0,0)$ and let $\boldsymbol{v}_2 =   \iota(i \boldsymbol{u}_1) = (0, 1, 0, 0)$.  
Fixing $r > 0$ and $\zeta \in [0, 2 \pi)$, 
let $\boldsymbol{w}_1 = \tfrac{1}{\sqrt{1 + r^2}}(1,0,r \cos \zeta , r \sin \zeta)$ and let $\boldsymbol{w}_2 = \tfrac{1}{\sqrt{1 + r^2}} (0, 1, - r \sin \zeta, r \cos \zeta)$.
Let $W = \Span\left\{ \boldsymbol{w}_1, \boldsymbol{w}_2\right\}$.  Then we have
\[
(x_1, x_2, y_1, y_2) \in N_{\delta}(W) \iff (x_1, x_2) \in B_{\R^2}\left( \tfrac{1}{r} R_{\zeta}(y_1, y_2), \frac{\delta \sqrt{1 + r^2}}{r} \right),
\]
where $R_{\zeta}(y_1, y_2) = (y_1 \cos \zeta + y_2 \sin \zeta, -y_1 \sin \zeta + y_2 \cos \zeta)$ is the rotation of $(y_1, y_2)$ through angle $\zeta$.
\end{lemma}

\begin{proof}
Let $\Pi_{W}$ denote the orthogonal projection onto $W$, and let $\Pi_{W^{\perp}}$ denote the orthogonal projection onto the orthogonal complement of $W$.
The point $(\vec{x}, \vec{y}) =  (x_1, x_2, y_1, y_2)$ is in $N_{\delta}(W)$ if and only if $\| \Pi_{W^{\perp}}(\vec{x}, \vec{y}) \| < \delta$.   We compute that
\[
\begin{split}
\| \Pi_{W^{\perp}}&(\vec{x}, \vec{y}) \|^2 =  \| (\vec{x}, \vec{y}) -  \Pi_{W^{\perp}}(\vec{x}, \vec{y}) \|^2  \\
&=  \| (\vec{x}, \vec{y}) -((\vec{x}, \vec{y}) \cdot \boldsymbol{w}_1 ) \boldsymbol{w}_1 - ( (\vec{x}, \vec{y}) \cdot \boldsymbol{w}_2) \boldsymbol{w}_2 \|^2  \\
&= \frac{1}{1 + r^2} \Big( (r^2 x_1^2 + r_2 x_2^2 + y_1^2 + y_2^2) \\
&\indent \indent \indent \indent - 2r ( (x_1y_1 + x_2 y_2) \cos \zeta + (x_1 y_2 - x_2 y_1) \sin \zeta) \Big) \\
&=  \frac{1}{1 + r^2} \Big( (rx_1 - (y_1 \cos \zeta + y_2 \sin \zeta) )^2 + (rx_2 - (-y_1 \sin \zeta + y_2 \cos \zeta) )^2 \Big) \\
&=  \frac{r}{1 + r^2} \Big( (x_1 - \tfrac{1}{r}(y_1 \cos \zeta + y_2 \sin \zeta) )^2 + (x_2 - \tfrac{1}{r}(-y_1 \sin \zeta + y_2 \cos \zeta) )^2  \Big) \\
&= \frac{r}{1 + r^2} \left \| \vec{x} - \tfrac{1}{r} R_{\zeta}(\vec{y}) \right \|^2.
\end{split}
\]
(Note that some steps are omitted between the second and third lines.)

We conclude that $\| \Pi_{W^{\perp}}(\vec{x}, \vec{y}) \| < \delta$ if and only if $(x_1, x_2)$ is within the ball of radius $\delta \sqrt{1 + r^2}/r$ centered at the point
\[
\tfrac{1}{r} R_{\zeta}(\vec{y}) = \tfrac{1}{r} (y_1 \cos \zeta + y_2 \sin \zeta, -y_1 \sin \zeta + y_2 \cos \zeta),
\]
as claimed.
\end{proof}

\subsection*{Acknowledgement}
\par Throughout the course of this work, S.T. has been supported by an NSF GRFP fellowship and by NSF DMS-2037851. We would like to thank our advisor, Larry Guth, for introducing us to Kakeya and related problems and supporting us through the research. Also, we would like to thank Dominique Maldague for her discussion about the initial formulation of the problem and her comments on an earlier draft, and the anonymous referee(s) for their many concrete suggestions.

\bibliographystyle{amsalpha}
\bibliography{9Biblio}

\end{document}